\documentclass[11pt,a4paper]{amsart}

\title[A pullback operation on a class of currents]{A pullback operation on a class of currents}

\author{H{\aa}kan Samuelsson Kalm}
\address{H{\aa}kan Samuelsson Kalm, Department of Mathematical Sciences, University of Gothenburg and 
Chalmers University of Technology, SE-412 96 G\"{o}teborg, Sweden}
\email{hasam@chalmers.se}

\date{\today}

\usepackage[english]{babel}
\usepackage{amsmath,calligra}
\usepackage{amsthm,amssymb,latexsym}
\usepackage[all,cmtip]{xy}
\usepackage{url,ae}
\usepackage{t1enc}
\usepackage{bm}
\usepackage{mathrsfs}
\usepackage{graphicx}
\usepackage{geometry}
\newtheorem{proposition}{Proposition}[section]
\newtheorem{theorem}[proposition]{Theorem}
\newtheorem{lemma}[proposition]{Lemma}

\theoremstyle{definition}
\newtheorem{definition}[proposition]{Definition}
\newtheorem{example}[proposition]{Example}
\newtheorem{remark}[proposition]{Remark}

\numberwithin{equation}{section}

\DeclareMathOperator{\Hom}{\mathscr{H}\text{\kern -3pt {\calligra\Large om}}\,}
\DeclareMathOperator{\Ext}{\mathscr{E}\text{\kern -3pt {\calligra\Large xt}}\,\,}
\DeclareMathOperator{\Image}{\mathscr{I}\text{\kern -3pt {\calligra\Large m}}\,}
\DeclareMathOperator{\Ker}{\mathscr{K}\text{\kern -3pt {\calligra\Large er}}\,}
\newcommand{\C}{\mathbb{C}}
\newcommand{\debar}{\bar{\partial}}

\newcommand{\J}{\mathcal{J}}

\newcommand{\PM}{\mathscr{P} \kern -3pt \mathscr{M}}

\newcommand{\PS}{\mathcal{PS}}

\newcommand{\CH}{\mathscr{C} \kern -2pt \mathscr{H}}

\def\newop#1{\expandafter\def\csname #1\endcsname{\mathop{\rm #1}\nolimits}}
\newop{span}

\begin{document}
\nocite{*}
\bibliographystyle{plain}

\begin{abstract}
For any holomorphic mapping $f\colon X\to Y$ between a complex manifold $X$ and a complex Hermitian manifold $Y$ we extend the pullback
$f^*$ from smooth forms to a class of currents. 
We provide a basic calculus for this pullback and show under quite mild assumptions that it is cohomologically sound. 
The class of currents we consider contains in particular the Lelong current of any
analytic cycle. Our pullback depends in general on the Hermitian structure of $Y$ but coincides with the usual pullback of currents in case 
$f$ is a submersion. The construction is 
based on the Gysin mapping in algebraic geometry. 
\end{abstract}

\maketitle
\thispagestyle{empty}

\section{Introduction}
Calculus of currents is central in complex geometry. One example is the celebrated Monge--Amp\`ere product. Closely related to 
current products is pullback of currents.
The purpose of this paper is to introduce 
a pullback operation on a class of currents 
and establish basic properties.

We say that a current $\mu$ on a complex manifold $Y$
is a $\PS$-\emph{current} (a pseudosmooth current), $\mu\in\mathcal{PS}(Y)$, 
if it locally is a finite sum of direct images $g_*\alpha$ where $\alpha$ is a smooth compactly supported form and $g$ is a holomorphic mapping. 
For instance, Lelong currents of analytic cycles are $\PS$-currents since if $i\colon V\to Y$ is a subvariety,
then $[V]=i_*1$, where $[V]$ is the Lelong current of integration over $V_{reg}$, and by 
a partition of unity on $V$ it follows that $[V]$ is a $\PS$-current. 
The class of $\PS$-currents is closed under exterior products with smooth forms and under $d$, $\debar$, and $\partial$.

\begin{theorem}\label{thm1}
Let $X$ be a complex manifold and let $Y$ be a complex Hermitian manifold. For any holomorphic mapping $f\colon X\to Y$ there is a 
linear mapping $f^*\colon\mathcal{PS}(Y)\to\mathcal{PS}(X)$ with the following properties:
If $\varphi$ is a smooth form on $Y$, then $f^{*}\varphi$ 
is the usual pullback. If additionally  $\mu\in\mathcal{PS}(Y)$, then
\begin{equation}\label{sov}
f^{*}(\varphi\wedge\mu)=f^*\varphi\wedge f^{*}\mu, 
\end{equation}
\begin{equation}\label{sov2}
f^{*}d\mu=d f^{*}\mu, \quad f^{*}\debar\mu=\debar f^{*}\mu, \quad f^{*}\partial\mu=\partial f^{*}\mu. 
\end{equation}

Let $U\subset X$ be an open set. Then $(f^*\mu)|_U = f|_U^*\mu$.
Assume that $f(U)$ is a complex manifold and that the mapping $\tilde f \colon U\to f(U)$ induced by $f$ is a submersion.
Let $\iota\colon f(U)\to Y$ be the inclusion so that $f|_U=\iota\circ\tilde f$. Then $f|_U^*\mu=\tilde f^*\iota^*\mu$,
where $\tilde f^*$ is the usual pullback of currents under a submersion. 
\end{theorem}

Notice in view of \eqref{sov} that $\text{supp}\, f^*\mu\subset f^{-1}(\text{supp}\, \mu)$.
Notice also that if $f$ is a submersion, then by the last part of Theorem~\ref{thm1}, our $f^*$
coincides with the usual pullback  of currents under a submersion.


For our second main result we will assume that $Y$ has an additional property.
We say that a complex manifold $Y$ is \emph{good} if there is a holomorphic section $\Phi$ of a holomorphic vector bundle $F\to Y\times Y$ 
such that $\Phi$ defines the diagonal in $Y\times Y$, i.e., $\Phi$ vanishes to first order precisely on the diagonal. 
Many complex manifolds are good.
For instance, all projective manifold are good, and any submanifold of a good manifold is good.

\begin{theorem}\label{thm2}
Let $Y$ be a good compact complex Hermitian manifold, $X$ a compact complex manifold, and $f\colon X\to Y$
a holomorphic mapping. If $\mu\in\mathcal{PS}(Y)$ is closed,
then for any Hermitian metric on $Y$, 
\begin{equation}\label{lus}
[f^*\mu]_{\text{dR}}=f^*[\mu]_{\text{dR}},
\end{equation} 
where $[\cdot]_{\text{dR}}$ means de~Rham cohomology class and $f^*$ in the right-hand side is the usual pullback
of cohomology classes. In particular, $[f^*\mu]_{dR}$ is independent of the Hermitian structure on $Y$.
\end{theorem}

The assumption that $Y$ is good is for proof-technical reasons and we believe that it is not essential for the conclusion of Theorem~\ref{thm2} to hold.

Theorem~\ref{thm2} implies that our pullback on closed $\PS$-currents is functorial on cohomology level. By this we mean that if
$f$ in Theorem~\ref{thm2} can be factorized as $f=f_2\circ f_1$,
where $f_1\colon X\to X'$ and $f_2\colon X'\to Y$ are holomorphic mappings and $X'$ is a good compact complex Hermitian manifold, then 
\begin{equation*}
[f^*\mu]_{dR}=[f_1^*f_2^*\mu]_{dR}.
\end{equation*} 
This follows indeed by Theorem~\ref{thm2} since $f^*[\mu]_{dR}=f_1^*f_2^*[\mu]_{dR}$.
However, in general  $f_1^*f_2^*\mu$ 
depends on a choice of Hermitian structure on $X'$ and therefore one cannot expect $f^*\mu=f_1^*f_2^*\mu$ to hold in general. 
Example~\ref{burger} below shows that in general $f^*\mu\neq f_1^*f_2^*\mu$. 

\smallskip

A standard approach to pullback of currents is the following. 
Let $X$, $Y$, and $f$ be as in Theorem~\ref{thm1}, let $\pi_1\colon X\times Y\to X$ and $\pi_2\colon X\times Y\to Y$
be the standard projections, and let $i\colon X\to X\times Y$ be the graph embedding defined by $i(x)=(x,f(x))$.
If $\mu$ is a current in $Y$, then $\pi_2^*\mu$ is well-defined since $\pi_2$ is a projection; in fact, $\pi_2^*\mu=1\otimes \mu$.
The main step is to give a reasonable meaning to the current product
\begin{equation}\label{hulan}
[i(X)]\wedge\pi_2^*\mu,
\end{equation} 
where $[i(X)]$ is the current of integration over $i(X)$. Then the pullback of $\mu$ under $f$ is defined as
\begin{equation}\label{motet}
(\pi_1)_*([i(X)]\wedge\pi_2^*\mu).
\end{equation}
If $\mu$ is a smooth form, then \eqref{hulan} is canonically defined and \eqref{motet} is the usual pullback.

We follow this approach and the novelty 
is our definition of \eqref{hulan}. 
The definition is modeled on the Gysin mapping introduced in \cite{AESWY1}.
If $\pi_2^*\mu$ is a (generalized) cycle and $i(X)$ can be defined by a global holomorphic section of some vector bundle over $X\times Y$, then \eqref{hulan} indeed
is the image of $\pi_2^*\mu$ under the Gysin mapping in \cite{AESWY1} associated to the embedding $i$.

In case $Y$ is good \eqref{hulan} can be defined as follows. Let $\Phi$ be a global holomorphic section of $F\to Y\times Y$ defining the diagonal.
Then $\Psi:=(f\times\text{id}_Y)^*\Phi$ is a holomorphic section of $E:=(f\times\text{id}_Y)^*\Phi$ over $X\times Y$ defining $i(X)$.
Let $N_X\to i(X)$ be the normal bundle of $i(X)$ in $X\times Y$. 
It is well-known that there is a canonical isomorphism $N_X\simeq \pi_2^*TY|_{i(X)}$. 
This isomorphism induces a Hermitian metric on $N_X$ since $TY$ is equipped with a Hermitian metric; by $Y$ being Hermitian, or having a Hermitian structure,
we mean that $TY$ is equipped with a Hermitian metric.
Moreover, $\Psi$ induces an embedding
$N_X\hookrightarrow E|_{i(X)}$ (see, e.g., \cite[Lemma~7.3]{AESWY1} or Section~\ref{segre-chern-normal} below) and we equip $E$ with a Hermitian metric so that this embedding
is Hermitian, i.e., is an embedding of Hermitian vector bundles.  It then follows from \cite[Proposition~1.5]{AESWY1}, see Example~\ref{grillad} and \eqref{cs1} below,
that we have
\begin{equation}\label{ultra}
[i(X)]=\hat c(N_X)\wedge M^\Psi,
\end{equation}
where $\hat c(N_X)$ is the full Chern form of $N_X$ and $M^\Psi$ 
is a certain current introduced 
by Andersson in \cite{MatsLelong}. 
The current $M^\Psi$ is the sum of currents $M^\Psi_k$ that can be 
defined as\footnote{In this paper $d^c=(\partial-\debar)/4\pi i$ so that 
if $z$ is the complex coordinate in $\C$ then $dd^c\log|z|^2=[0]$.}
$$
M_k^\Psi=\lim_{\epsilon\to 0} \debar\chi(|\Psi|^2/\epsilon)\wedge\frac{\partial \log|\Psi|^2}{2\pi i}\wedge (dd^c\log|\Psi|^2)^{k-1}
$$
if $\chi$ is a smooth regularization of the characteristic function of $[1,\infty)\subset\mathbb{R}$.
These currents are in fact the restriction to $i(X)$ of Monge--Amp\`{e}re products $(dd^c\log|\Psi|^2)^k$, where 
$(dd^c\log|\Psi|^2)^k$ are the Monge--Amp\`{e}re products of Bedford--Taylor and Demailly if $k\leq \text{codim}\, i(X)=\text{dim}\, Y$,
and were introduced in
\cite{MatsLelong} if $k> \text{codim}\, i(X)$.

A basic observation now is that if $\mu\in\PS(Y)$, then one can give a natural meaning to the
products $M_k^\Psi\wedge\pi_2^*\mu$. Indeed, if $\chi$ is as above, then it turns out that the limits
\begin{equation}\label{vadd}
M_k^\Psi\wedge\pi_2^*\mu:=\lim_{\epsilon\to 0} \debar\chi(|\Psi|^2/\epsilon)\wedge\frac{\partial \log|\Psi|^2}{2\pi i}\wedge (dd^c\log|\Psi|^2)^{k-1}\wedge\pi_2^*\mu, \quad k=1,2,\ldots
\end{equation}
exist, are independent of the choice of $\chi$, and are in $\PS(X\times Y)$. Clearly, $M_k^\Psi\wedge\pi_2^*\mu$ have support in $i(X)$.
The limit 
\begin{equation}\label{vaddaa}
M_0^\Psi\wedge\pi_2^*\mu:=\lim_{\epsilon\to 0} (1-\chi(|\Psi|^2/\epsilon))\pi_2^*\mu
\end{equation}
also turns out to exist and to have the properties just mentioned for the limit \eqref{vadd}.
We then define \eqref{hulan} 
to be the component of $\hat c(N_X)\wedge M^\Psi\wedge\pi_2^*\mu$ of the ``expected bidegree'' $(\text{dim}\,Y,\text{dim}\, Y)+\text{bidegree}(\mu)$.
In general, $\hat c(N_X)\wedge M^\Psi\wedge\pi_2^*\mu$ has components of various different bidegrees, but 
if $\mu$ is a smooth form, then in view of \eqref{ultra}, $\hat c(N_X)\wedge M^\Psi\wedge\pi_2^*\mu$ is the canonical product $[i(X)]\wedge\pi_2^*\mu$, which has bidegree
$(\text{dim}\,Y,\text{dim}\, Y)+\text{bidegree}(\mu)$.
We can now define $f^*\mu$ by \eqref{motet}. Since we are taking the component of the expected bidegree in the definition of \eqref{hulan},
it follows that $f^*\mu$ has the same bidegree as $\mu$. After a simplification (using \eqref{fotball} below)
we get the formula
\begin{equation}\label{alt12}
f^*\mu=\sum_{k=0}^nf^*\hat c_{n-k}(TY)\wedge (\pi_{1})_*(M_k^\Psi\wedge\pi_2^*\mu).
\end{equation}

As an illustration, let $f\colon Bl_0\C^2\to\C^2$ be the blowup of the origin in $\C^2$ and let $\mu$ be the Dirac measure at $0\in\C^2$. 
With the standard Hermitian structure on $\C^2$ we have
$f^*\mu=\omega\wedge [D]$, where $D\simeq\mathbb{P}^1$ is the exceptional divisor and $\omega$ is the standard Fubini-Study metric form on 
that $\mathbb{P}^1$, see Example~\ref{strutt} below. 

In this example so-called dimensional excess occurs; 
$\pi_2^*\mu$ and the graph of $f$ do not intersect properly. 
In such situations it is reasonable that $f^*\mu$ has to depend on some additional structure, e.g., a Hermitian metric as in our case. 
This can be compared to the intersection theory in \cite{Fulton} where the intersection of two cycles, which do not intersect properly, 
is determined only up to rational equivalence.



\smallskip

As already mentioned, Lelong currents of analytic cycles are $\PS$-currents but $\PS$-currents also appear naturally in the study of
Monge--Amp\`ere products associated to holomorphic sections of Hermitian vector bundles. For instance, in that setting the so-called  
non-pluripolar part is a $\PS$-current.
Moreover, the pullback of a smooth form under a meromorphic mapping is usually not smooth but only
a $\PS$-current. Therefore our class of currents may have relevance in the study of the dynamics of such mappings.
We believe that the complete generality of the mapping $f\colon X\to Y$ and the properties of our $f^*$ described in 
Theorems~\ref{thm1} and \ref{thm2}
can complement and provide a new viewpoint on existing works on pullback of currents.
Let us mention a few such works.

The intersection theory in \cite{Fulton} contains in particular a 
quite general theory of pullback of cycles such that if $Z$ is a cycle in $Y$ then $f^*(Z)$ is a rational equivalence class in $X$. 
This is extended to pullback of Green currents in the context of arithmetic intersection theory by  Gillet--Soul\'e in \cite{GS1990}.
In complex geometry mainly pullback of positive closed currents has been considered.  
A positive closed $(1,1)$-current locally has a $dd^c$-potential, which is a plurisubharmonic function.
By using (quasi-)plurisubharmonic $dd^c$-potentials one can define the pullback of such currents under surjective mappings, see M\'eo \cite{meo}.
Dinh and Sibony have considered pullback
of positive closed $(p,p)$-currents in several papers. In \cite{DS2007} is defined an essentially canonical pullback, 
with good continuity and cohomological properties,
of such currents under 
holomorphic mappings with constant fiber dimension, cf.\ Example~\ref{sib} below.
In \cite{DSacta} and \cite{DS2010} a general theory of super-potentials
is developed on compact K\"{a}hler manifolds generalizing the case of $(1,1)$-currents. 
Building on some works by Dinh and Sibony, Truong defines a pullback under dominant meromorphic mappings \cite{Tuyen}.
Rather recently, in \cite{DS2018}, on 
K\"{a}hler manifolds was introduced the notions of density currents and 
tangent currents to a positive closed current 
along a submanifold. We believe that the tangent currents to $\pi_2^*\mu$, where as before $\pi_2\colon X\times Y\to Y$is the natural projection,
 along the graph of $f$ is closely connected to our approach; 
cf.\ \cite{KV} where Kaufmann and Vu compare density currents to Monge--Amp\`ere-type products.
We also mention the recent paper \cite{Barlet-beta} by Barlet; on any reduced pure-dimensional analytic space is introduced 
a sheaf of $\debar$-closed $(p,0)$-currents, closed under wedge products and the de~Rham differential, 
so that the pullback under a holomorphic mapping, whose image is not contained in the singular locus, is well-defined and functorial. 


The plan of the paper is as follows. Preliminaries and some basic facts about $\PS$-currents are collected in 
Section~\ref{prelim}. Even though we are primarily interested in currents on complex manifolds,  
it is convenient and natural in our context to recall some basic facts about 
differential forms and currents on (reduced) analytic spaces.
In Section~\ref{MPS} we consider the operator $\mu\mapsto M^\Psi\wedge\mu$ on $\PS$-currents.
This operator is used in Section~\ref{gysinsektion}
to adapt the Gysin mapping in \cite{AESWY1} to our setting. Using this Gysin mapping the pullback operation is introduced in Section~\ref{pb-op-sec} 
and Theorem~\ref{thm1}
is proved. In Section~\ref{diagonal-alternativ} we give an alternative definition of our pullback that can be generalized
to give a definition of a pullback under meromorphic mappings and meromorphic correspondences.
Theorem~\ref{thm2} is proved in Section~\ref{cohom} and functoriality is discussed in Section~\ref{obstruktion}.

\smallskip

{\bf Acknowledgment:} I would like to thank Mats Andersson and Nessim Sibony for valuable comments.
I would also like to thank the anonymous referee for careful reading and important comments and suggestions.


\section{Preliminaries and $\PS$-currents}\label{prelim}
Let $Z$ be a (reduced) analytic space of pure dimension. 
Locally $Z$ can be embedded as an analytic subset of an open set $\Omega$ of some $\C^N$. 
A smooth differential form $\varphi$ in $Z_{reg}$ is smooth in $Z$, $\varphi\in\mathcal{E}(Z)$, 
if there is a local embedding $\iota\colon Z\to\Omega$ and a smooth form $\tilde\varphi$ in $\Omega$
such that $\varphi=\iota^*\tilde\varphi$ on $Z_{reg}$. 
It is well-known that this notion of smooth forms in $Z$ is independent of the choice of local embedding.
It follows that $d$, $\debar$, and $\partial$ are well-defined on $\mathcal{E}(Z)$.

The space of currents, $\mathcal{C}(Z)$, in $Z$
is the dual of the space of test forms, i.e., compactly supported smooth forms in $Z$.
More concretely, if $\iota\colon Z\to\Omega$ is a local embedding, then the currents in $Z$ can be identified with
the currents in $\Omega$ vanishing on test forms $\xi$ with $i^*\xi=0$ in $Z_{reg}$.
By duality, $d$, $\debar$, and $\partial$ are well-defined on $\mathcal{C}(Z)$.
As usual we also have $\mathcal{E}(Z)\subset\mathcal{C}(Z)$ and that $\mathcal{C}(Z)$
is a module over $\mathcal{E}(Z)$.

Let $V$ be an analytic space and $g\colon V\to Z$ a holomorphic mapping. By \cite[Corollary~3.2.21]{BM}
there is a natural pullback mapping $g^*\colon \mathcal{E}(Z)\to \mathcal{E}(V)$. 
If $g$ is proper it follows that there is a pushforward mapping
$g_*\colon\mathcal{C}(V)\to \mathcal{C}(Z)$
defined by $\langle g_*\mu,\xi\rangle=\langle\mu,g^*\xi\rangle$. 
If $\mu\in\mathcal{C}(V)$ and $\varphi\in\mathcal{E}(Z)$, then we have
the following projection formula
\begin{equation}\label{projformel}
\varphi\wedge g_*\mu = g_*(g^*\varphi\wedge\mu).
\end{equation}
If $\mu\in\mathcal{C}(V)$ has compact support, then $g_*\mu$ is defined (in the same way)
also if $g$ is not proper, and \eqref{projformel}
holds.

If $\pi\colon V\to Z$ is a submersion, so that $\pi_*$ of test forms in $V$ are test forms in $Z$,
and $\mu\in\mathcal{C}(Z)$, then $\pi^*\mu$ is defined by $\langle \pi^*\mu,\xi\rangle:=\langle \mu,\pi_*\xi\rangle$.
In particular, $\pi_*\mu$ is defined if $\pi\colon W\times Z\to Z$ is the standard projection. In this case we have
\begin{equation}\label{rast}
\pi^*\mu=1\otimes\mu.
\end{equation}

\begin{definition}\label{PSdef}
A current $\mu$ in $Z$ is in $\PS(Z)$ if in a neighborhood of each point, $\mu$ is a finite sum of currents
$g_*\alpha$, where $g\colon V\to Z$ is a holomorphic mapping from a connected complex manifold $V$ and 
$\alpha$ is a smooth form with compact support in $V$.
\end{definition}

It is clear from the definition that $\PS$-currents have order $0$, and
by \eqref{projformel} it is clear that $\PS(Z)$ is an $\mathcal{E}(Z)$-module. Since $\bar\alpha$ is smooth if $\alpha$
is smooth, and $\overline{g_*\alpha}=g_*\bar\alpha$ it follows that 
\begin{equation}\label{psbarps}
\overline{\PS(Z)} = \PS(Z).
\end{equation}
Moreover, it is clear that if $\widetilde Z$ is a pure dimensional analytic space and $h\colon \widetilde Z\to Z$ is a 
proper holomorphic mapping, then 
\begin{equation}\label{datalabb}
h_*\colon \PS(\widetilde Z)\to\PS(Z).
\end{equation}

One can let $V$ in Definition~\ref{PSdef} be an analytic space without changing $\PS(Z)$. In fact, if $V$ is
an analytic space and $\alpha$ is a smooth form in $V$, then one can first assume that $V$ is irreducible by restricting $\alpha$ to each irreducible component. 
By Hironaka's theorem then there is a modification $\pi\colon \widetilde V\to V$ such that $\widetilde V$
is a connected manifold.
Since $\pi$ is a modification,
\begin{equation}\label{penna}
\alpha = \pi_*\pi^*\alpha
\end{equation}
as currents, and so $g_*\alpha=g_*\pi_*\pi^*\alpha$. Hence, one can replace $V$ by $\widetilde V$ 
and $\alpha$ by $g^*\alpha$. By this observation we get

\begin{example}\label{lelongex}
Let $V\subset Z$ be an analytic subset and $i\colon V\to Z$ the inclusion. 
By Lelong's theorem there is a closed current $[V]$ in $Z$ 
defined as integration over $V_{reg}$. The function $1$ on $V$ is smooth and $i_*1=[V]$. Thus,
by a smooth partition of unity on $V$, we see that
$
[V]\in\PS(Z).
$

It follows that if $\nu=\sum_ja_jV_j$ is an analytic cycle in $Z$, then the Lelong current $[\nu]=\sum_ja_j[V_j]$ is in
$\PS(Z)$. Notice also that there is a $\mu\in\PS(|\nu|)$ such that $[\nu]=j_*\mu$, where $j\colon |\nu|\to Z$
is the inclusion of the support of $\nu$.

Generalized cycles, as introduced in \cite{AESWY1}, are in $\PS(Z)$ as well.\qed
\end{example}

\smallskip

We are primarily interested in $\PS$-currents on manifolds. We will see that on a manifold, $\PS$-currents are \emph{pseudomeromorphic}, so
we recall the definition and some properties that we will use.
Pseudomeromorphic currents were introduced 
in \cite{AWdecomp} and further developed in \cite{ASdolb}.

In $\C$ we have the principal value currents $1/z^\ell$ and the associated residue currents $\debar (1/z^\ell)$.
An \emph{elementary pseudomeromorphic current} is a current in an open set $U\subset\C^N$ of the form
$$
\alpha\wedge\frac{1}{z_1^{\ell_1}}\cdots \frac{1}{z_r^{\ell_r}}\debar\frac{1}{z_{r+1}^{\ell_{r+1}}}\wedge\cdots\wedge\debar\frac{1}{z_{s}^{\ell_{s}}},
$$
where $\alpha$ is a smooth form with compact support in $U$, $(z_1,\ldots,z_N)$ are coordinates in $U$, and
the current products are tensor products. A germ of a current $\tau$ at a point in $Z$ is \emph{pseudomeromorphic} if
it is a finite sum of currents of the form
$$
\pi_*^1\pi_*^2\cdots\pi_*^k\nu,
$$
where each $\pi^j\colon V_j\to V_{j-1}$ is either a modification, a simple projection $V_j\times W\to V_j$, or the inclusion of an open 
subset, and $\nu$ is an elementary pseudomeromorphic current in $V_k\subset\C^N$. The set of such germs turns out to be a sheaf, 
the sheaf $\mathcal{PM}_Z$ of pseudomeromorphic currents in $Z$. We refer to \cite{AW4} for properties of
pseudomeromorphic currents.

The sheaf $\mathcal{PM}_Z$ is closed under $\debar$, $\partial$, and 
under multiplication by smooth forms. Moreover, we have the

\smallskip

\noindent \emph{Dimension principle:} If $\tau\in\mathcal{PM}(Z)$ has bidegree $(*,q) $ and support in an analytic subset
with codimension $>q$, then $\tau=0$.

\begin{example}
Let $\sigma$ be a holomorphic section of a line bundle $L$ and let $\alpha$ be a smooth form with values in $L$.
Then the semi-meromorphic form $\alpha/\sigma$ has a canonical extension across $\{\sigma=0\}$ 
as a pseudomeromorphic current; it can be obtained for instance as a principal value. \qed
\end{example}

If $\tau\in\mathcal{PM}(Z)$ and $V\subset Z$ is a subvariety, then the restriction of $\tau$ to the open set
$Z\setminus V$ has an extension to a pseudomeromorphic current $\mathbf{1}_{Z\setminus V}\tau$ in $Z$ such that
\begin{equation}\label{restr}
\mathbf{1}_{Z\setminus V}\tau = \lim_{\epsilon\to 0}\chi(|h|^2v/\epsilon) \tau
\end{equation}
if $\chi$ is a smooth approximation of the characteristic function of $[1,\infty)\subset\mathbb{R}$, $h$ is a holomorphic 
tuple with $\{h=0\}=V$, and $v$ is a smooth positive function.
It follows that
\begin{equation}\label{restrV}
\mathbf{1}_V\tau:=\tau-\mathbf{1}_{Z\setminus V}\tau
\end{equation}
is pseudomeromorphic with support in $V$. If $\alpha$ is a smooth form, then $\mathbf{1}_V(\alpha\wedge\tau)=\alpha\wedge\mathbf{1}_{V}\tau$.
If $g\colon \widetilde Z\to Z$ proper, $\nu\in\mathcal{PM}(\widetilde Z)$, and $g_*\nu\in\mathcal{PM}(Z)$, then
\begin{equation}\label{restr2}
\mathbf{1}_Vg_*\nu = g_*(\mathbf{1}_{g^{-1}V}\nu).
\end{equation}
If $\nu$ has compact support, then the same holds for any holomorphic mapping $g$.

A current $a$ on $Z$ is \emph{almost semi-meromorphic}, $a\in ASM(Z)$, if $a=\pi_*(\alpha/\sigma)$, where $\pi\colon\widetilde Z\to Z$ is a modification,
$\sigma$ is a holomorphic section of a line bundle $L\to\widetilde Z$, and $\alpha$ is a smooth form with values in $L$; see, e.g., \cite[Section~4]{AW4}.
We have that $ASM(Z)\subset\mathcal{PM}(Z)$. If $a\in ASM(Z)$, then the smallest Zariski closed set outside which $a$ is smooth is called the
Zariski singular support of $a$.

%
%
%
%
%
%
%
%
%
%
%
%
%

\begin{lemma}[\cite{AW4}, Theorem~4.8]\label{asmpm}
If $a\in ASM(Z)$ has Zariski singular support $V$ and $\tau\in \mathcal{PM}(Z)$,
 then there is a unique $T\in\mathcal{PM}(Z)$ such that $T=a\wedge\tau$ in $Z\setminus V$
and $\mathbf{1}_VT=0$.
\end{lemma}

We will denote the current $T$ by $a\wedge\tau$. In view of \eqref{restr} and \eqref{restrV}, since $\mathbf{1}_V(a\wedge\tau)=0$ we have
\begin{equation}\label{asmpm2}
a\wedge\tau = \lim_{\epsilon\to 0} \chi(|h|^2v/\epsilon) a\wedge\tau
\end{equation}
if $\chi$, $h$, and $v$ are as above.

\begin{lemma}[\cite{AW4}, Theorem~2.25] \label{kontakt}
If $Z$ is a complex manifold, then a germ of a current $\mu$ at a point in $Z$ is pseudomeromorphic if and only if
it is a finite sum of currents of the form $h_*\tau$, where $h\colon U\to Z$ is a  holomorphic mapping and $\tau$ is elementary. 
\end{lemma}

In view of this lemma and Definition~\ref{PSdef} it is clear that if $Z$ is a manifold, then $\PS(Z)\subset\mathcal{PM}(Z)$. On a manifold thus $\PS$-currents have the 
properties of pseudomeromorphic currents. 
In particular, restrictions as in \eqref{restr} and \eqref{restrV} are defined on $\PS$-currents.
Let us check that such restrictions preserve $\PS(Z)$. It suffices to see that
if $V\subset Z$ is a subvariety and $\mu\in\PS(Z)$, then $\mathbf{1}_V\mu\in\PS(Z)$. To see this we can assume that $\mu=g_*\alpha$,
where $g$ and $\alpha$ are as in Definition~\ref{PSdef}. By \eqref{restr2} thus 
$\mathbf{1}_V\mu=g_*(\mathbf{1}_{g^{-1}V}\alpha)$. Since $\alpha$ is smooth we have $\mathbf{1}_{g^{-1}V}\alpha=0$
if $g^{-1}V$ is a proper subvariety, and $\mathbf{1}_{g^{-1}V}\alpha=\alpha$ otherwise.
Hence, if $\mu=g_*\alpha$, then either $\mathbf{1}_V\mu$ is $0$ or $\mu$, both 
of which are in $\PS(Z)$.

\subsection{Segre forms, Chern forms, and normal bundles}\label{segre-chern-normal}
We first give a brief presentation of Segre forms and Chern forms based on the presentation in \cite[Section~2]{AESWY1}.

Recall that if $L\to Z$ is a Hermitian line bundle, then there is the associated first Chern form $\hat c_1(L)$, which is a smooth closed $(1,1)$-form.
The first Chern class, $c_1(L)$, of $L$ is the de~Rham cohomology class of $\hat c_1(L)$ and is independent of the Hermitian metric on $L$.
Let $E\to Z$ be a Hermitian vector bundle. Over the projectivization $\pi\colon \mathbb{P}(E)\to Z$ of $E$
(the projective bundle of lines through the zero section of $E$) we have 
the tautological line bundle $L=\mathcal{O}_E(-1)\subset \pi^*E$ and we equip it with the induced Hermitian metric. 
The \emph{total Segre form}, $\hat s(E)=1+\hat s_1(E)+\hat s_2(E)+\cdots$, is defined by
$$
\hat s(E)=\pi_*(1/(1+\hat c_1(L)))=\pi_*(1/(1-\hat c_1(L^*))) .
$$
If $\text{rank}\, E=n$ it follows that
\begin{equation}\label{vafflor}
\hat s_\ell(E)=\pi_* \hat c_1(L^*)^{\ell+n-1}. 
\end{equation}
If $n=1$, then $\mathbb{P}(E)= Z$ so that $\pi=\text{id}_X$, and hence $(1+\hat c_1(L))\wedge\hat s(E)=1$.
In the general case we define
the \emph{total Chern form}, $\hat c(E)=1+\hat c_1(E)+\hat c_2(E)+\cdots$, of $E$ by
\begin{equation}\label{cs1}
\hat c(E)\wedge\hat s(E)=1.
\end{equation}
It is proved in \cite{mourougane} that this definition coincides with the differential geometric definition. 
It is well-known that $\hat c_\ell(E)=0$ if $\ell>n$, and that if $h\colon\widetilde Z\to Z$ is a holomorphic mapping, then
$h^*\hat c(E)=\hat c(h^*E)$ and $h^*\hat s(E)=\hat s(h^*E)$. The Chern classes $c_j(E)$ are the de~Rham 
cohomology classes of $\hat c_j(E)$ and they are independent of the Hermitian metric on $E$.

\smallskip

We will primarily consider Chern and Segre forms of normal bundles of submanifolds, and some times also of locally complete intersections,
so we recall a few things that we need about normal bundles.
Let $\J\subset\mathcal{O}_Z$ be a locally complete intersection ideal sheaf of codimension $n$ and zero set $X$. We will use the following ad hoc definition
of the normal bundle $N_\J\to X$, cf.\ \cite[Section~7]{AESWY1}. 
A section $\phi$ of $N_\J$ 
is a choice of holomorphic $n$-tuple $\phi(s)$ locally in $X$ for each local minimal set
$s=(s_1,\ldots,s_n)$ of generators for $\J$ such that $A\phi(s)=\phi(As)$ on $X$
for any locally defined holomorphic matrix $A$ that is invertible in a neighborhood of $X$.
Thus, each local minimal set of generators for $\J$ gives a local trivialization of $N_\J$.

Assume now that $\J$ is generated by a holomorphic section $\Psi$ of a vector bundle $E\to Z$. Then we have an induced embedding
\begin{equation}\label{elsa}
N_\J\hookrightarrow E|_X
\end{equation}
defined as follows, see, e.g., \cite[Section~7]{AESWY1}. Let $s$ be a local minimal set of generators for $\J$.
Since $\Psi$ generates $\J$ there is a local holomorphic section $B$ of $\text{Hom}(\C^n,E)$ such that 
$\Psi=Bs$ and $B|_X$ is unique and pointwise injective; here $\C^n$ is the trivial vector bundle of rank $n$ over $Z$. 
The embedding \eqref{elsa} then is given by
\begin{equation}\label{elsas}
\phi\mapsto B|_X\phi(s).
\end{equation}

If $\J$ is principal, then $\J$ defines a divisor, $D$, and $\Psi=\Psi^0\Psi'$, where $\Psi^0$ is a holomorphic section of the line bundle $L$ corresponding to $D$
such that $\text{div}\,\Psi^0=D$ and $\Psi'$ is a holomorphic section of  $L^*\otimes E$. In this case $N_\J=L|_{|D|}$
and the embedding \eqref{elsa} extends to an embedding
\begin{equation}\label{rov}
L\hookrightarrow E, \quad \sigma\mapsto \sigma\Psi'.
\end{equation}
For the induced metric $|\cdot |_L$ on $L$ thus $|\Psi^0|_L=|\Psi|$ and so, by the Poincar\'e--Lelong formula,
\begin{equation}\label{pl}
dd^c\log|\Psi|^2 = [D] - \hat c_1(L).
\end{equation}

If $Z$ is a complex manifold and $X\subset Z$ is a submanifold we let $\J_X\subset\mathcal{O}_Z$ be the sheaf of holomorphic functions vanishing on $X$
and we write $N_X$ instead of $N_{\J_X}$. 

If $h\colon \widetilde Z\to Z$ is a holomorphic mapping and $\J\subset \mathcal{O}_Z$, then we let  
$h^*\J\subset\mathcal{O}_{\widetilde Z}$ be the sheaf generated by $h^*$ of the generators of $\J$.
The following functorial property should be well-known but for the reader's convenience we supply a proof. 

\begin{lemma}\label{sladd}
Let $Z$ and $\widetilde Z$ be complex manifolds and $X\subset Z$ a complex submanifold. Assume that $h\colon \widetilde Z\to Z$ is a holomorphic mapping
such that $\widetilde \J:=h^*\J_X$  is a locally complete intersection with zero set $\widetilde X\subset \widetilde Z$.
\begin{itemize}
\item[(a)] There is a natural embedding $N_{\widetilde \J}\hookrightarrow h|_{\widetilde X}^*N_X$; it is an equality if $\text{codim}\,\widetilde X = \text{codim}\, X$.

\item[(b)] Assume that $\J_X$ is generated by a holomorphic section $\Psi$ of a vector bundle $E\to Z$,
so that we have an induced embedding $N_X\hookrightarrow E|_X$.
Then the composition of the embeddings
\begin{equation}\label{sladda}
N_{\widetilde \J}\hookrightarrow h|_{\widetilde X}^* N_X \hookrightarrow h|_{\widetilde X}^*E|_X,
\end{equation}
where the first one is the embedding in (a) and the second one is the pullback of $N_X\hookrightarrow E|_X$,
is the embedding $N_{\widetilde \J}\hookrightarrow (h^*E)|_{\widetilde X}$ induced by $h^*\Psi$.

\item[(c)] Assume that $\J_X$ is generated by a holomorphic section $\Psi$ of $E$ as in (b).
If $E$, $N_X$, and $N_{\widetilde \J}$ are equipped with Hermitian metrics such that 
the embeddings $N_X\hookrightarrow E|_X$ and $N_{\widetilde \J}\hookrightarrow (h^*E)|_{\widetilde X}$ are Hermitian,
then the embeddings in \eqref{sladda} are Hermitian.
\end{itemize}
\end{lemma}

\begin{proof}
Let $\tilde s=(\tilde s_1,\ldots,\tilde s_\kappa)$ be a local minimal set of generators of $\widetilde\J$
and $s=(s_1,\ldots,s_n)$ a local minimal set of generators of $\J_{X}$.

(a): Since by assumption $h^*s$ generates $\widetilde\J$ it follows that $\kappa\leq n$ and that there is a local holomorphic matrix $A$
such that $h^*s=A\tilde s$ and $\text{rank}\, A|_{\widetilde X}=\kappa$. We define the embedding 
$N_{\widetilde \J}\hookrightarrow h|_{\widetilde X}^*N_X$  by the local embeddings
\begin{equation}\label{kalas}
\phi\mapsto A|_{\widetilde X}\phi(\tilde s),
\end{equation}
where $\phi$ is a local section of $N_{\widetilde \J}$. The right-hand side is independent of the local trivialization
of $N_{\widetilde \J}$ given by $\tilde s$ and transforms as a section of $h|_{\widetilde X}^*N_X$.
Thus, the local embeddings \eqref{kalas} indeed give a global embedding $N_{\widetilde \J}\hookrightarrow h|_{\widetilde X}^*N_X$.

Since $n=\text{codim}\,X$ and $\kappa=\text{codim}\,\widetilde X$ it follows that if $\text{codim}\,X=\text{codim}\,\widetilde X$, then
\eqref{kalas} is an isomorphism.

(b): We have that $s$ gives a local trivialization of $N_X$ and hence of $h|^*_{\widetilde X}N_X$. 
Moreover, $\Psi=Bs$ for a local holomorphic section $B$ of $\text{Hom}(\C^n,E)$. In view of \eqref{elsas},
the second embedding in \eqref{sladda} is defined by
\begin{equation}\label{kalasa}
\xi(s)\mapsto h^*B|_{\widetilde X} \xi(s),
\end{equation}
where $\xi(s)$ is a local section of $h|^*_{\widetilde X}N_X$ in the trivialization given by $s$.

We also have that $h^*\Psi=C\tilde s$ for a local holomorphic section $C$ of $\text{Hom}(\C^\kappa,h^*E)$
such that $C|_{\widetilde X}$ is unique. The embedding $N_{\widetilde \J}\hookrightarrow (h^*E)|_{\widetilde X}$
is given by
\begin{equation}\label{kalasat}
\phi\mapsto C|_{\widetilde X}\phi(\tilde s).
\end{equation}
Since $\Psi=Bs$ and $h^*s=A\tilde s$ we have $h^*\Psi=h^*Bh^*s=h^*BA\tilde s$. By uniqueness thus
\begin{equation}\label{flervarr}
C|_{\widetilde X} = h^*B|_{\widetilde X}A|_{\widetilde X}
\end{equation}
and it follows that \eqref{kalasat} is the composition of 
\eqref{kalas} and \eqref{kalasa}.

(c): If the embedding $N_X\hookrightarrow E|_X$ is Hermitian, then it is clear that 
the second embedding in \eqref{sladda}, which is given by \eqref{kalasa}, is Hermitian. 
If the embedding $N_{\widetilde \J}\hookrightarrow (h^*E)|_{\widetilde X}$ is Hermitian it thus follows 
from \eqref{kalasat}, \eqref{flervarr}, and \eqref{kalasa} that
\begin{equation*}
|\phi|^2_{N_{\widetilde \J}} = \big|C|_{\widetilde X}\phi(\tilde s)\big|^2_{h^*E}
= \big|h^*B|_{\widetilde X}A|_{\widetilde X}\phi(\tilde s)\big|^2_{h^*E}
= \big|A|_{\widetilde X}\phi(\tilde s)\big|^2_{h^*N_X}.
\end{equation*}
By \eqref{kalas} thus the first embedding in \eqref{sladda} is Hermitian.
\end{proof}

\subsection{The $M$-operator on pseudomeromorphic currents}\label{MPM}
Let $\Psi$ be a holomorphic section of a Hermitian vector bundle $E\to Z$ with zero set $X$.
We recall the following lemma from \cite[Section~2]{prog}.

\begin{lemma}\label{lillam}
There are unique almost semi-meromorphic currents $m_k^\Psi$, $k=1,2,\ldots$, in $Z$ that coincide 
with $(2\pi i)^{-1}\partial\log|\Psi|^2\wedge (dd^c\log|\Psi|^2)^{k-1}$ in $Z\setminus X$
and such that $\mathbf{1}_X m_k^\Psi=0$.
\end{lemma}

In view of Lemma~\ref{asmpm}, for $\tau\in\mathcal{PM}(Z)$ we have $m_k^\Psi\wedge\tau\in\mathcal{PM}(Z)$ and we let
\begin{equation}\label{gurkmajo}
M_0^\Psi\wedge\tau:=\mathbf{1}_X\tau,\quad M_k^\Psi\wedge\tau:=\mathbf{1}_X\debar(m_k^\Psi\wedge\tau),\,\, k=1,2,\ldots.
\end{equation}
Since restriction and $\debar$ preserve $\mathcal{PM}$ we have $M^\Psi_k\wedge\tau\in\mathcal{PM}(Z)$.
Let $\chi_\epsilon=\chi(|\Psi|^2/\epsilon)$, where $\chi$ is as in \eqref{restr}. In view of \eqref{restr}, \eqref{restrV}, and \eqref{asmpm2},
\begin{eqnarray*}
\mathbf{1}_X\debar(m_k^\Psi\wedge\tau) &=& \lim_{\epsilon\to 0} (1-\chi_\epsilon)\debar(m_k^\Psi\wedge\tau)
=\lim_{\epsilon\to 0} \debar((1-\chi_\epsilon) m_k^\Psi\wedge\tau) + \debar\chi_\epsilon\wedge (m_k^\Psi\wedge\tau) \\
&=&
\debar(\mathbf{1}_X m_k^\Psi\wedge\tau) + \lim_{\epsilon\to 0} \debar\chi_\epsilon \wedge (m_k^\Psi\wedge\tau).
\end{eqnarray*}
By Lemmas~\ref{asmpm} and \ref{lillam} we have $\mathbf{1}_X m_k^\Psi\wedge\tau=0$  and so we get, cf.\ \eqref{vadd} and \eqref{vaddaa},
\begin{equation}\label{universeum}
M_0^\Psi\wedge\tau=\lim_{\epsilon\to 0}(1-\chi(|\Psi|^2/\epsilon))\tau, \quad
M_k^\Psi\wedge\tau=\lim_{\epsilon\to 0}\debar\chi(|\Psi|^2/\epsilon)\wedge m_k^\Psi\wedge\tau.
\end{equation}

Let $M^\Psi\wedge\tau=M_0^\Psi\wedge\tau+M_1^\Psi\wedge\tau+\cdots$. If $\varphi$ is a smooth form in $Z$ it follows by \eqref{universeum} that 
\begin{equation}\label{korv}
M^\Psi\wedge(\varphi\wedge\tau) = \varphi\wedge M^\Psi\wedge\tau.
\end{equation}
If $g\colon\widetilde Z\to Z$ is a holomorphic mapping and $\mu\in\mathcal{PM}(\widetilde Z)$ are such that $g_*\mu$ is defined and in $\mathcal{PM}(Z)$,
then by \eqref{universeum} and \eqref{projformel},
\begin{equation}\label{mos}
M^\Psi\wedge g_*\mu = g_*(M^{g^*\Psi}\wedge\mu).
\end{equation}

We will write $M^\Psi$ instead of $M^\Psi\wedge 1$. Notice that if $\Psi$ is generically non-vanishing, then $M^\Psi_0=0$.

\begin{example}\label{grillad}
Assume that $\Psi$ defines a locally complete intersection ideal sheaf $\J$. 
Then by \eqref{elsa} there is an embedding $N_\J\hookrightarrow E|_X$ and we equip $N_\J$ with the induced metric.
By \cite[Proposition~1.5]{AESWY1},
\begin{equation*}
M^\Psi = \hat s(N_\J)\wedge [\mathcal{Z}_\J],
\end{equation*}
where $\mathcal{Z}_\J$ is the fundamental cycle of the locally complete intersection $\J$; see, e.g., \cite[Chapter~1.5]{Fulton}.
We note two special cases:
\begin{itemize}
\item[(i)] If $\J$ is the ideal sheaf of a submanifold $X$, then $\mathcal{Z}_\J=X$ and so $M^\Psi= \hat s(N_\J)\wedge [X]$.

\item[(ii)] If $\J$ is principal so that it defines a divisor $D$, then $\mathcal{Z}_\J=D$. In view of \eqref{vafflor} thus
$M_k^\Psi = \hat c_1(L^*)^{k-1}\wedge [D]$
if $L$ is the line bundle associated with $D$ equipped with the metric induced by \eqref{rov}. \qed
\end{itemize} 
\end{example}

In the following, if $\Psi$ generates the ideal sheaf of a submanifold $X$, then we will usually simply say that $\Psi$ defines $X$.
Similarly, if $\Psi$ generates the principal ideal sheaf of a divisor $D$ we say that $\Psi$ defines $D$.

\section{The $M$-operator on $\PS$-currents}\label{MPS}
Let $Z$ be a complex manifold. Then $\PS(Z)\subset\mathcal{PM}(Z)$ by Lemma~\ref{kontakt}. Hence,
 if $\mu\in\PS(Z)$ and $\Psi$ is a holomorphic section of a Hermitian vector bundle, then
$M^\Psi\wedge\mu$ is defined and has the properties in Section~\ref{MPM}.
In this section we will see that $M^\Psi\wedge\mu$ has some additional properties when $\mu\in\PS(Z)$. Lemmas~\ref{senf} and \ref{gurka} below are straightforward 
adaptions of parts of \cite[Theorem~5.2]{AESWY1} to $\PS$-currents. 

\begin{lemma}\label{senf}
Let $\Psi$ be a holomorphic section of a Hermitian vector bundle $E\to Z$. 
Let $X=\{\Psi=0\}$, let $i\colon X\to Z$ be the inclusion, and let $\text{codim}\, X=n$. 
If $\mu\in\PS(Z)$, then there are unique $\hat s_{k-n}(\Psi,\mu)\in\PS(X)$
such that
\begin{equation*}
M_k^\Psi\wedge\mu = i_* \hat s_{k-n}(\Psi,\mu), \quad\quad k=0,1,2,\ldots.
\end{equation*}
\end{lemma}

\begin{proof}
The uniqueness is clear since $i_*$ is injective. It therefore suffices to show the lemma locally in $Z$. 
We can thus assume that $\mu=g_*\alpha$,
where $\alpha$ and $g\colon V\to Z$ are as in Definition~\ref{PSdef}.
By Hironaka's theorem, after a modification we can assume that the ideal sheaf generated by $g^*\Psi$ is principal, cf.\ \eqref{penna}.
Let $D$ be the divisor defined by $g^*\Psi$ and let $L$ be the associated line bundle 
equipped with the metric induced by the embedding $L\hookrightarrow g^*E$, cf.\ \eqref{rov}. 
Then by \eqref{mos},  \eqref{korv}, and Example~\ref{grillad}, 
\begin{equation}\label{maja}
M_k^\Psi\wedge\mu = M_k^\Psi\wedge g_*\alpha = g_*(M_k^{g^*\Psi}\wedge\alpha)
=g_*(\alpha\wedge M_k^{g^*\Psi}) = g_*\big(\alpha \wedge\hat c_1(L^*)^{k-1}\wedge [D]\big).
\end{equation}

Let $W=|D|$, let $j\colon W\to V$ be the inclusion, and
consider the fiber diagram
\begin{equation}\label{eq:tomte}
\xymatrix{
W  \ar[d]^h \ar[r]^j & V \ar[d]^g\\
X \ar[r]^i & Z.
}
\end{equation}
By Example~\ref{lelongex} there is a $\mu_{D}\in\PS(W)$ such that 
$[D]=j_*\mu_{D}$. Hence, by \eqref{maja}, \eqref{projformel}, and commutativity of \eqref{eq:tomte},
\begin{equation}\label{marion}
M_k^\Psi\wedge\mu = g_*j_*\big(j^*\alpha\wedge j^*\hat c_1(L^*)^{k-1} \wedge\mu_{D}\big)
=i_*h_*\big(j^*\alpha\wedge j^*\hat c_1(L^*)^{k-1} \wedge\mu_{D}\big).
\end{equation}
Since the class of $\PS$-currents is closed under multiplication by smooth forms and under direct images of 
proper holomorphic mappings it follows that 
\begin{equation*}
\hat s_{k-n}(\Psi,\mu):=h_*\big(j^*\alpha\wedge j^*\hat c_1(L^*)^{k-1} \wedge\mu_{D}\big)
\end{equation*}
is in $\PS(X)$. The lemma thus follows by \eqref{marion}.
\end{proof}

It follows from Lemma~\ref{senf} and \eqref{datalabb} that $M^\Psi\wedge\mu\in\PS(Z)$ if $\mu\in\PS(Z)$.

\begin{lemma}\label{gurka}
Let $\Psi$ be a holomorphic  section of a Hermitian vector bundle $E\to Z$. If $\mu\in\PS(Z)$, then 
$d(M^\Psi\wedge\mu) = M^\Psi\wedge d\mu$. The same holds with $d$ replaced by $\debar$ or $\partial$.
\end{lemma}
\begin{proof}
This is a local statement so we can assume that $\mu=g_*\alpha$, where $\alpha$ 
and $g\colon V\to Z$ are as in Definition~\ref{PSdef}. By Hironaka's theorem, in view of \eqref{penna}, we can also assume that $g^*\Psi$ 
defines a principal ideal sheaf. By Example~\ref{grillad} thus $M^{g^*\Psi}$ is $d$-, $\debar$-, and
$\partial$-closed. In view of \eqref{mos} and \eqref{korv}, using that  $d$ and $g_*$ commute and that $M^{g^*\Psi}$ is closed,
we get
\begin{equation*}
d(M^\Psi\wedge\mu) = d(M^\Psi\wedge g_*\alpha) = g_*d(M^{g^*\Psi}\wedge\alpha)
= g_*(M^{g^*\Psi}\wedge d\alpha) = M^\Psi\wedge g_*d\alpha = M^\Psi\wedge d\mu.
\end{equation*}
The same calculation can be done with $d$ replaced by $\debar$ or $\partial$.
\end{proof}

In general, if $X\subset Z$ is a complex submanifold one cannot expect there to be a global holomorphic section $\Psi$
of a vector bundle $E\to Z$ defining $X$.
However, for any complex submanifold $X\subset Z$ the next result allows us to define global $\PS$-currents 
$\hat s_j(N_X,\mu)$ in $X$ for $\mu\in\PS(Z)$ generalizing $\hat s_j(\Psi,\mu)$ in Lemma~\ref{senf}
as soon as the normal bundle $N_X$ of $X$ is equipped with a Hermitian metric.
This
generalizes \cite[Definition~5.5]{AESWY1} to the setting of $\PS$-currents
and when there is no global holomorphic section defining $X$.

\begin{proposition}\label{ketchup}
Let $Z$ be a complex manifold, $X\subset Z$ a complex submanifold of codimension $n$, and $i\colon X\to Z$ the inclusion.
Assume that the normal bundle $N_X\to X$ of $X$ is equipped with a Hermitian metric. If $\mu\in\PS(Z)$,
then there are unique $\hat s_{k-n}(N_X,\mu)\in\PS(X)$ with the following property:
If there is a holomorphic section $\Psi$ of a Hermitian vector bundle $E$ in an open $U\subset Z$ such that $\Psi$ defines $X$ in $U$,
and the induced embedding $N_X\hookrightarrow E|_X$ is Hermitian, then in $U$
\begin{equation}\label{smaskens}
i_*\hat s_{k-n}(N_X,\mu)=M^\Psi_k\wedge\mu, \quad k=0,1,2\ldots.
\end{equation} 
\end{proposition}

\begin{proof}
Locally in $Z$ there are $\Psi$ and $E$ such that $\Psi$ defines $X$
and $N_X\hookrightarrow E|_X$ is Hermitian. For instance, take local coordinates $z$ such that $X=\{z_1=\cdots =z_n=0\}$,
let $E$ be the trivial rank $n$-bundle, let $\Psi=(z_1,\ldots,z_n)$, and choose a Hermitian metric on 
$E$ appropriately. Since $i_*$ is injective on currents it follows that $\hat s_{k-n}(N_X,\mu)$ must be unique.

To show the existence of $\hat s_{k-n}(N_X,\mu)$ it thus suffices to show it locally. Locally
in $Z$ we can assume that $\mu=g_*\alpha$, where $\alpha$ and $g\colon V\to Z$ are as in Definition~\ref{PSdef}.
We can also assume, by Hironaka's theorem and \eqref{penna}, that $g^*\J_X$ is principal.
Let $D$ be the divisor defined by $g^*\J_X$ and let $L$ be the associated line bundle. Recall that $L|_{|D|}$ is the normal
bundle of $g^*\J_X$. By Lemma~\ref{sladd} (a) we have an embedding
\begin{equation}\label{troca}
L|_{|D|}\hookrightarrow g^*N_X
\end{equation}
and we equip $L$ with a Hermitian metric so that \eqref{troca} is Hermitian. Consider the fiber diagram \eqref{eq:tomte} with 
$W=|D|$. By Example~\ref{lelongex} there is a $\mu_D\in\PS(|D|)$ such that $j_*\mu_D=[D]$, and we let 
\begin{equation}\label{dero}
\hat s_{k-n}(N_X,\mu):=h_*(\mu_D\wedge j^*\hat c_1(L^*)^{k-1}\wedge j^*\alpha).
\end{equation}
Since $\mu_D\in\PS(|D|)$, $j^*\hat c_1(L^*)^{k-1}\wedge j^*\alpha$ is smooth, and $h\colon |D|\to X$ is a proper holomorphic mapping
it follows that $\hat s_{k-n}(N_X,\mu)\in\PS(X)$. 

It remains to check that $\hat s_{k-n}(N_X,\mu)$, as defined in \eqref{dero}, have the claimed property.
Assume therefore that $\Psi$ is a holomorphic section of a Hermitian vector bundle $E$ that defines $X$ in $U$ and such that the 
induced embedding $N_X\hookrightarrow E|_X$ is Hermitian.  In view of \eqref{mos}, \eqref{korv}, and
Example~\ref{grillad}, 
\begin{equation}\label{champis}
M_k^\Psi\wedge\mu = g_*(M_k^{g^*\Psi}\wedge\alpha) = g_*([D]\wedge\hat c_1(L^*)^{k-1}\wedge\alpha),
\end{equation}
where $\hat c_1(L^*)$ here is with respect to the metric on $L$ induced by the embedding $L\hookrightarrow g^*E$ given by $g^*\Psi$; cf.\ \eqref{rov}. 
By Lemma~\ref{sladd}, on 
$W$ this metric is the same as the one induced by \eqref{troca}. Since $[D]=j_*\mu_D$, by \eqref{champis},
commutativity of \eqref{eq:tomte}, and \eqref{dero}, thus
\begin{equation*}
M_k^\Psi\wedge\mu = g_*j_*(\mu_D\wedge j^*\hat c_1(L^*)^{k-1}\wedge j^*\alpha)
=i_*g_*(\mu_D\wedge j^*\hat c_1(L^*)^{k-1}\wedge j^*\alpha)
=i_*\hat s_{k-n}(N_X,\mu).
\end{equation*}
This concludes the proof.
\end{proof}

\begin{example}\label{unna}
In the situation of Proposition~\ref{ketchup}, if $\mu$ is a smooth form, then 
\begin{equation}\label{tunna}
\hat s_{k-n}(N_X,\mu)=\hat s_{k-n}(N_X)\wedge i^*\mu. 
\end{equation}
In particular, if $X$ is a hypersurface with associated line bundle $L$, then $\hat s_{k-1}(N_X,\mu)=\hat c_1(L|_{X}^*)^{k-1}\wedge i^*\mu$
since then $N_X=L|_X$.
The equality \eqref{tunna} follows from \eqref{smaskens}, \eqref{korv}, and Example~\ref{grillad} since locally there always are $\Psi$ and $E$ as in 
Proposition~\ref{ketchup}. \qed
\end{example}

\begin{proposition}\label{besikta}
Let $Z$ be a complex manifold and $\Psi$ a holomorphic section of a Hermitian vector bundle $E\to Z$.
Assume that $\Psi$ defines a divisor $D$ and let $L$ be the associated line bundle equipped with the Hermitian metric
induced by \eqref{rov}. 
If $\mu\in\PS(Z)$ is closed and of positive degree, then for $k\geq 1$, there are currents $\nu_k$ in $Z$ such that
\begin{equation*}
M_k^\Psi\wedge\mu = -\hat c_1(L^*)^k\wedge\mathbf{1}_{Z\setminus |D|}\mu +
d\nu_k.
\end{equation*}
\end{proposition}

\begin{proof}
Let $\chi_\epsilon=\chi(|\Psi|^2/\epsilon)$. By \eqref{universeum}, a straightforward calculation, and using that $d\mu=0$, we get
\begin{eqnarray*}
M_k^\Psi\wedge\mu &=& 
\lim_{\epsilon\to 0} \debar\chi_\epsilon\wedge\frac{\partial\log|\Psi|^2}{2\pi i}\wedge (dd^c\log|\Psi|^2)^{k-1}\wedge\mu \\
&=& 
\lim_{\epsilon\to 0} d\chi_\epsilon\wedge d^c\log|\Psi|^2\wedge (dd^c\log|\Psi|^2)^{k-1}\wedge\mu \\
&=&
\lim_{\epsilon\to 0} -\chi_\epsilon (dd^c\log|\Psi|^2)^{k}\wedge\mu
+ d\big(\chi_\epsilon d^c\log|\Psi|^2\wedge (dd^c\log|\Psi|^2)^{k-1}\wedge\mu\big).
\end{eqnarray*}
Using \eqref{pl} and \eqref{restr} thus
\begin{eqnarray*}
M_k^\Psi\wedge\mu &=& \lim_{\epsilon\to 0} -\chi_\epsilon (-\hat c_1(L))^{k}\wedge\mu
+ d\big(\chi_\epsilon d^c\log|\Psi|^2\wedge (-\hat c_1(L))^{k-1}\wedge\mu\big) \\
&=&  
-\hat c_1(L^*)^{k}\wedge \mathbf{1}_{Z\setminus |D|}\mu
+ \lim_{\epsilon\to 0} d\big(\chi_\epsilon d^c\log|\Psi|^2\wedge \hat c_1(L^*)^{k-1}\wedge\mu\big).
\end{eqnarray*}

We claim that there are current $\nu_k$ in $Z$ such that
$$
\chi_\epsilon d^c\log|\Psi|^2\wedge \hat c_1(L^*)^{k-1}\wedge\mu\to\nu_k, \quad\epsilon\to 0.
$$
Taking the claim for granted the proposition immediately follows. 
To show the claim, since $\hat c_1(L^*)$ is smooth it suffices to see that $\lim_{\epsilon\to 0}\chi_\epsilon d^c\log|\Psi|^2\wedge\mu$ exists.
By Lemma~\ref{kontakt}, since $Z$ is a manifold, $\mu$ is pseudomeromorphic.
It therefore follows by Lemma~\ref{lillam}, Lemma~\ref{asmpm}, and \eqref{asmpm2} that
\begin{equation*}
\lim_{\epsilon\to 0} \chi_\epsilon \partial\log|\Psi|^2\wedge\mu
\end{equation*}
exists and is a pseudomeromorphic current. By \eqref{psbarps} and Lemma~\ref{kontakt}
also $\bar\mu$ is pseudomeromorphic, and so
\begin{equation*}
\lim_{\epsilon\to 0} \chi_\epsilon \debar\log|\Psi|^2\wedge\mu
= \lim_{\epsilon\to 0}  \overline{\chi_\epsilon\partial\log|\Psi|^2\wedge\bar\mu}
\end{equation*}
exists and is the conjugate of a pseudomeromorphic current. Since $d^c=(\partial-\debar)/4\pi i$ 
it follows that $\lim_{\epsilon\to 0}\chi_\epsilon d^c\log|\Psi|^2\wedge\mu$ exists, which shows the claim.
\end{proof}



\section{Gysin mappings}\label{gysinsektion}
Let $i\colon X\to Z$ be an embedding of a complex $m$-dimensional manifold $X$ into an $m+n$-dimensional complex manifold $Z$. 
Suppose that the normal bundle $N_{X}\to i(X)$ of the submanifold $i(X)$ is equipped with a Hermitian metric.
Notice that $\text{rank}\, N_X=\text{codim}\, i(X)=n$. 

Associated with the embedding $i\colon X\to Z$ there is a Gysin mapping 
$\mathcal{A}_k(Z)\to\mathcal{A}_{k-n}(X)$, where $\mathcal{A}_k$ is the Chow group
of $k$-cycles modulo rational equivalence; see, e.g., \cite[Section~6]{Fulton}. 
In \cite{AESWY1} are introduced generalized cycles, $\mathcal{GZ}$,
which are $\PS$-currents,
and a certain quotient group 
$\mathcal{B}$ of $\mathcal{GZ}$,
which can be thought of as an analogue of the Chow group. 
If $i(X)$ can be defined by a global holomorphic section of some vector bundle,
then there are also Gysin mappings $\mathcal{GZ}_k(Z)\to\mathcal{GZ}_{k-n}(X)$
and $\mathcal{B}_k(Z)\to\mathcal{B}_{k-n}(X)$
analogous to the one in \cite{Fulton}.

Let $\mu\in\PS(Z)$. Recall from Proposition~\ref{ketchup} that there are $\hat s_{k-n}(N_X,\mu)\in\PS(X)$ since $N_X$ is equipped
with a Hermitian metric. We define Gysin mappings 
$i^!\colon \PS(Z)\to\PS(X)$ and $i^{!!}\colon \PS(Z)\to\PS(X)$ 
by
\begin{equation}\label{gysin}
i^!\mu = \sum_{k=0}^n i^*\hat c_{n-k}(N_X)\wedge\hat s_{k-n}(N_X,\mu),
\end{equation}
\begin{equation}\label{fullgysin}
i^{!!}\mu = i^*\hat c(N_X)\wedge\hat s(N_X,\mu).
\end{equation}
The mapping \eqref{gysin} preserves bidegree and
is the basis of our pullback. 
It is straightforward to check that \eqref{gysin} is the component of 
the ``full'' Gysin mapping \eqref{fullgysin} of the same bidegree as $\mu$; 
this makes \eqref{fullgysin} useful in some calculations. 

\begin{example}\label{lokal}
Suppose that there is a global holomorphic section $\Psi$ of a Hermitian vector bundle $E\to Z$ defining $i(X)$
and that the induced embedding $N_X\hookrightarrow E|_{i(X)}$ (cf.\ \eqref{elsa}) is Hermitian. Then
in view of Proposition~\ref{ketchup} and \eqref{projformel}, 
\begin{equation}\label{Mgysin}
i_*i^!\mu = \sum_{k=0}^n \hat c_{n-k}(N_X)\wedge M_k^\Psi\wedge\mu,
\end{equation}
\begin{equation}\label{fullMgysin}
i_*i^{!!}\mu = \hat c(N_X)\wedge M^\Psi\wedge\mu.
\end{equation}
It follows that
the restriction of $i^!$ to 
$\mathcal{GZ}(Z)$ is the Gysin mapping in \cite[Eq.\ (1.9)]{AESWY1}.

Recall from the proof of Proposition~\ref{ketchup} above that locally in $Z$ there always are $\Psi$ and $E$ with $\Psi$ defining $i(X)$. \qed
\end{example}

\begin{example}\label{kaffe}
Let $\mu\in\PS(Z)$ and assume that $\text{supp}\,\mu\subset i(X)$. Then 
\begin{equation}\label{salt}
 \mu=i_*\hat s_{-n}(N_X,\mu) \quad \text{and} \quad
i^!\mu = i^*\hat c_n(N_X)\wedge\hat s_{-n}(N_X,\mu).
\end{equation}
These are local statements so
we can assume that there are $\Psi$ and $E$ as in Example~\ref{lokal}. 
Since $\text{supp}\,\mu\subset i(X)$ it follows by \eqref{gurkmajo} and \eqref{universeum} that
\begin{equation*}
M_0^\Psi\wedge\mu=\mathbf{1}_X\mu=\mu, \quad \quad M_k^\Psi\wedge\mu=0, k\geq 1.
\end{equation*}
Thus \eqref{salt} follows in view of \eqref{gysin} and Proposition~\ref{ketchup}. \qed
\end{example}

\begin{proposition}\label{gyProp1}
The Gysin mappings  $i^{!}$ and $i^{!!}$ are linear mappings $\mathcal{PS}(Z)\to\mathcal{PS}(X)$
and commute with $d$, $\debar$, and $\partial$.
If $\varphi$ is a smooth form in $Z$, then $i^{!}\varphi=i^{!!}\varphi=i^*\varphi$ and 
if $\mu\in\mathcal{PS}(Z)$, then
\begin{equation}\label{hoho}
i^{!} (\varphi\wedge\mu)=i^*\varphi\wedge i^{!} \mu, \quad
i^{!!} (\varphi\wedge\mu)=i^*\varphi\wedge i^{!!} \mu.
\end{equation}
\end{proposition}

\begin{proof}
By Proposition~\ref{ketchup}, $\mu\mapsto \hat s_{k-n}(N_X,\mu)$ are mappings $\mathcal{PS}(Z)\to\mathcal{PS}(X)$.
By \eqref{gysin} and \eqref{fullgysin}, since $\PS(X)$ is closed under multiplication by smooth forms,
$i^{!}$ and $i^{!!}$ are mappings $\mathcal{PS}(Z)\to\mathcal{PS}(X)$ as well.

Locally in $Z$ there are $\Psi$ and $E$ as in Example~\ref{lokal}; we will use this
to show the rest of the statements of the proposition.
Since $\mu\mapsto M_k^\Psi\wedge\mu$
are linear it follows from Proposition~\ref{ketchup}
that $\mu\mapsto \hat s_{k-n}(N_X,\mu)$ are linear. Thus $i^{!}$ and $i^{!!}$ are linear.
Moreover, using \eqref{fullMgysin}, that Chern forms are closed, and  Lemma~\ref{gurka} we have 
\begin{equation*}
i_*di^{!!}\mu = di_*i^{!!}\mu=d(\hat c(N_X)\wedge M^\Psi\wedge\mu)
=\hat c(N_X)\wedge M^\Psi\wedge d\mu = i_*i^{!!}d\mu.
\end{equation*}
Since $i_*$ is injective it follows that $di^{!!}\mu=i^{!!}d\mu$. Then $di^{!}\mu=i^{!}d\mu$ follows by taking the component of the
right bidegree. The same calculations can be done with $d$ replaced by $\debar$ and $\partial$.

Let $\varphi$ be a smooth form in $Z$. Then by \eqref{korv}, Example~\ref{grillad}, and \eqref{projformel} 
\begin{equation*}
M^\Psi\wedge\varphi = \varphi\wedge M^\Psi\wedge 1 = \varphi\wedge \hat s(N_X)\wedge [X] 
=\varphi\wedge \hat s(N_X)\wedge i_*1 = \hat s(N_X)\wedge i_*i^*\varphi.
\end{equation*}
By \eqref{fullMgysin} and \eqref{cs1} thus,
\begin{equation*}
i_*i^{!!}\varphi = \hat c(N_X)\wedge M^\Psi\wedge\varphi
=\hat c(N_X)\wedge\hat s(N_X)\wedge i_*i^*\varphi= i_*i^*\varphi. 
\end{equation*}
Since $i_*$ is injective it follows that $i^{!!}\varphi=i^*\varphi$ and so, by taking the component of the right bidegree,  $i^{!}\varphi=i^*\varphi$.
If $\mu\in\PS(Z)$, by \eqref{korv} and \eqref{projformel} we have
\begin{equation*}
i_*i^{!!}(\varphi\wedge\mu) = \hat c(N_X)\wedge M^\Psi\wedge (\varphi\wedge\mu)
=\varphi \wedge \hat c(N_X)\wedge M^\Psi\wedge\mu = \varphi\wedge i_*i^{!!}\mu =i_*(i^*\varphi\wedge i^{!!}\mu).
\end{equation*}
It follows that $i^{!!}(\varphi\wedge\mu)=i^*\varphi\wedge i^{!!}\mu$ and so $i^{!}(\varphi\wedge\mu)=i^*\varphi\wedge i^{!}\mu$
as before. This completes the proof.
\end{proof}

Since by this proposition $i^!\varphi=i^*\varphi$  if $\varphi$ is a smooth form in $Z$,
it follows that $i_*i^!\varphi=[i(X)]\wedge \varphi$. For a general $\mu\in\PS(Z)$ we
define
\begin{equation}\label{diagprod}
[i(X)]\wedge\mu:=i_*i^!\mu.
\end{equation}

\section{The pullback operation}\label{pb-op-sec}
Throughout this section we use the following notation and setup.
Let $X$ be a complex $m$-dimensional manifold, $Y$ a complex $n$-dimensional Hermitian manifold, and
$f\colon X\to Y$ a holomorphic mapping. 
We let $Z=X\times Y$ and $\pi_1\colon Z\to X$ and $\pi_2\colon Z\to Y$ are
the natural projections. Let also $i\colon X\to Z$, $i(x)=(x,f(x))$, be the graph embedding.
The normal bundle $N_X\to i(X)$ of $i(X)$ in $Z$ is naturally isomorphic with $TY$ via
\begin{equation}\label{roffe}
N_X=\pi_2^* TY|_{i(X)}.
\end{equation}
Since $Y$ is Hermitian, $TY$ has a Hermitian metric and we equip $N_X$ with the metric induced by \eqref{roffe}.

\subsection{Definition and basic properties}
Let $\mu\in\mathcal{PS}(Y)$. Since $\pi_2$ is a projection, $\pi_2^*\mu=1\otimes \mu$ is a well-defined current in $Z$; cf.\ \eqref{rast}. 
Let us check that $\pi_2^*\mu\in\mathcal{PS}(Z)$. This is a local statement so we can assume that $\mu=g_*\alpha$, 
where $g\colon V\to Y$ and $\alpha$ are as in Definition~\ref{PSdef}. 
Let $\{\rho_j\}$ be a locally finite partition of unity on $X$ with $\rho_j$ smooth and compactly supported.
Then 
\begin{equation}\label{lakan}
\pi_2^*\mu=1\otimes \mu=\sum_j\rho_j\otimes\mu=\sum_j (\text{id}_X\times g)_* (\rho_j\otimes \alpha),
\end{equation} 
which shows that $\pi_2^*\mu\in\mathcal{PS}(Z)$. We thus have the product
$[i(X)]\wedge\pi_2^*\mu$. Since $\pi_1\circ i=\text{id}_X$ it follows from
\eqref{diagprod} that 
$(\pi_1)_*([i(X)]\wedge\pi_2^*\mu)=(\pi_1)_*i_*i^!\pi_2^*\mu=i^!\pi_2^*\mu$. 
We use the last expression as the definition of $f^*$, cf.\ \eqref{motet}.

\begin{definition}\label{ister}
For $\mu\in\PS(Y)$ we let $f^*\mu=i^{!}\pi_2^*\mu$, 
where $i^!$ 
is the Gysin mapping \eqref{gysin}.
\end{definition}


We also introduce a ``full'' pullback mapping, $f^\diamond$, using the full Gysin mapping \eqref{fullgysin} by 
\begin{equation}\label{urk}
f^{\diamond}\mu = i^{!!}\pi_2^*\mu.
\end{equation}
Since $i^!\pi_2^*\mu$ is the component of $i^{!!}\pi_2^*\mu$ of the same bidegree as $\pi_2^*\mu$ it follows that
$f^*\mu$ is the component of $f^{\diamond}\mu$ of the same bidegree as $\mu$. 
This makes $f^{\diamond}$ convenient to use in some calculations. 


\begin{proposition}\label{kruka}
The operation $f^*$ is a linear mapping $\mathcal{PS}(Y)\to\mathcal{PS}(X)$ and commutes with $d$, $\debar$, and $\partial$.
If $\varphi$ is a smooth form in $Y$ and $\mu\in\mathcal{PS}(Y)$, then $f^*\varphi$ 
is the usual pullback and
\begin{equation}\label{uggla}
f^* (\varphi\wedge\mu) = f^*\varphi\wedge f^*\mu. 
\end{equation}
\end{proposition}

\begin{proof}
In view of \eqref{rast}, $\pi_2^*$ is a linear mapping $\mathcal{PS}(Y)\to\mathcal{PS}(Z)$ and so it follows from Proposition~\ref{gyProp1} that
$f^*$ 
is a linear mapping $\mathcal{PS}(Y)\to\mathcal{PS}(X)$.
Moreover, 
$\pi_2^*d\mu=d\pi_2^*\mu$ 
and similarly for $\debar$ and $\partial$. 
By Proposition~\ref{gyProp1} thus  $f^*$ 
commutes with $d$, $\debar$, and $\partial$.

Let $\varphi$ be a smooth form in $Y$. Then $\pi_2^*\varphi$ is smooth in $Z$ and so, by Proposition~\ref{gyProp1},
\begin{equation*}
i^!\pi_2^*\varphi=i^*\pi_2^*\varphi=(\pi_2\circ i)^*\varphi. 
\end{equation*}
Since $\pi_2\circ i=f$ thus $i^!\pi_2^*\varphi=f^*\varphi$ is the usual pullback.
In view of \eqref{rast} we have $\pi_2^*(\varphi\wedge\mu)=\pi_2^*\varphi\wedge\pi_2^*\mu$. Hence, \eqref{uggla}
follows from \eqref{hoho}.
\end{proof}

We will now see that if $Y$ is good, then \eqref{alt12} holds. 
Recall that $Y$ is good means that  there is a holomorphic section $\Phi$ of a holomorphic vector bundle
$F\to Y\times Y$ defining the diagonal $\Delta\subset Y\times Y$.
Then $\Psi:=(f\times \text{id}_Y)^*\Phi$ is a holomorphic section of 
$E:=(f\times \text{id}_Y)^*F$ defining $i(X)$ in $Z$.
We choose a Hermitian metric on $E$ such that the induced embedding 
$
N_X\hookrightarrow  E|_{i(X)},
$
is a Hermitian. 
Hence, when $Y$ is good we are in the setting of Example~\ref{lokal}.
Using that $\text{id}_X=\pi_1\circ i$ it thus follows from \eqref{Mgysin} that
\begin{equation}\label{kattmisse}
f^*\mu = i^!\pi_2^*\mu = (\text{id}_X)_*i^!\pi_2^*\mu  = (\pi_1)_* i_*i^!\pi_2^*\mu
= (\pi_1)_*\sum_{k=0}^n \hat c_{n-k}(N_X)\wedge M_k^\Psi\wedge\pi_2^*\mu.
\end{equation}
Now, $\pi_2|_{i(X)}=f\circ\pi_1|_{i(X)}$ so we can replace $\pi_2^*$ in \eqref{roffe} by $\pi_1^*f^*$.
Since the Hermitian metric on $N_X$ is induced by \eqref{roffe} it follows by functoriality of Chern forms that 
\begin{equation}\label{fotball}
\hat c(N_X) =  \pi_1^*f^*\hat c(TY)|_{i(X)}.
\end{equation}
Replacing $\hat c_{n-k}(N_X)$ in \eqref{kattmisse} by $\pi_1^*f^*\hat c_{n-k}(TY)$ and 
using \eqref{projformel} thus \eqref{alt12} follows.
In the same way one can check that 
\begin{equation}\label{alt11}
f^\diamond\mu = f^*\hat c(TY)\wedge (\pi_1)_*\big(M^\Psi\wedge\pi_2^*\mu).
\end{equation}

\begin{example}\label{strutt}
Let $Y=\C^2$, with $TY$ equipped with the standard Hermitian metric, and let $f\colon X\to Y$ be the blowup of $0$.
Let us show that if $\mu$ is the Dirac mass at $0$ (considered as a $(2,2)$-current),
then
\begin{equation}\label{pokemon}
f^*\mu=\omega\wedge [D], \quad \quad f^\diamond\mu=\omega\wedge [D] +[D],
\end{equation}
where $D\simeq \mathbb{P}^1$ is the exceptional divisor of the blowup and $\omega$ is the standard Fubini--Study
metric form on that $\mathbb{P}^1$. 

Notice first that $\pi_2^*\mu=g_*1$, where $g\colon X\to X\times Y$ is the mapping $g(x)=(x,0)$.
We consider the tuple $\Psi=f(x)-y$ as a section of the trivial rank-$2$ bundle $E\to X\times Y$
and equip $E$ with the standard Hermitian metric. Clearly, $\Psi$ defines the graph of $f$ in $X\times Y$.
One can check that the metric induced on $N_X$ by the embedding $N_X\hookrightarrow E|_{i(X)}$ 
is the same as the one induced by
$N_X=\pi_2^*TY|_{i(X)}$. We can thus use 
\eqref{alt11} to calculate 
$f^\diamond\mu$.
Notice that $\hat c(TY)=1$ since $TY$ has the standard metric. By \eqref{alt11} thus
\begin{equation}\label{grisfot}
f^\diamond\mu=(\pi_1)_*\big(M^\Psi\wedge\pi_2^*\mu).
\end{equation}
Since $\pi_2^*\mu=g_*1$, $g^*\Psi=f(x)$, and $f$ defines $D$ it follows from \eqref{projformel}
and Example~\ref{grillad} that
\begin{eqnarray}\label{sylta}
(\pi_{1})_*(M^\Psi\wedge\pi_2^*\mu) &=& (\pi_{1})_*g_* (M^{g^*\Psi}\wedge 1) = (\pi_{1})_*g_* M^{f(x)} \\
&=&
(\pi_{1})_*g_* \big([D] + [D]\wedge \hat c_1(L^*)\big), \nonumber
\end{eqnarray}
where $L$ is the line bundle corresponding to $D$ equipped with the metric induced by $L\hookrightarrow g^*E$.
The embedding $L|_D\hookrightarrow g^*E|_D$ is the standard embedding of $\mathcal{O}(-1)$ into 
$\mathbb{P}^1\times\mathbb{C}^2$ and so $\hat c_1(L^*)|_D=\omega$. Since 
$\pi_1\circ g=\text{id}_X$, the second equality in \eqref{pokemon} thus follows 
from \eqref{grisfot} and \eqref{sylta}. The first equality in \eqref{pokemon} then follows by taking the right bidegree. \qed
\end{example}

\begin{remark}
Let $f\colon X\to Y$ be the blowup of $0\in\mathbb{C}^2=Y$ 
and $\mu$ the Dirac mass at $0$ as in Example~\ref{strutt}. 
Then no reasonable pullback of currents under $f$ can be continuous. To see this, let $a_j$ be a sequence in $Y$
such that $a_j\to 0$. If $\mu^{a_j}$ is the Dirac measure at $a_j$, then $\mu^{a_j}\to\mu$ as currents.
Since $f$ is a biholomorphism outside $f^{-1}(0)$ we have that $f^*\mu^{a_j}$ is the Dirac mass at $f^{-1}(a_j)$ 
if $f^*$ is a reasonable pullback.
If $a_j$ and $b_j$ are sequences going to $0$ along different lines in $Y$, then $f^{-1}(a_j)$ and $f^{-1}(b_j)$
have different limits in $X$. Hence, $\mu^{a_j}\to\mu$ and $\mu^{b_j}\to\mu$, but $f^*\mu^{a_j}$ and $f^*\mu^{b_j}$ have different limits. \hfill $\qed$
\end{remark}

\begin{example}\label{ex:id}
Let  $f\colon X\to Y$ be the inclusion of an open set $X\subset Y$. Then, for 
any $\mu\in\PS(Y)$ we have $f^*\mu=f^{\diamond}\mu=\mu|_X$. 
To see this, we will show that
\begin{equation}\label{lur}
i_*f^{\diamond}\mu=i_*(\mu|_X);
\end{equation} 
recall that $i\colon X\to Z$ is the graph embedding, which in this case is $i(x)=(x,x)$.
Since $i_*$ is injective we get $f^{\diamond}\mu=\mu|_X$ from \eqref{lur}. Then $f^*\mu=\mu|_X$
follows by taking the right bidegree.


It suffices to verify \eqref{lur} in a neighborhood of a point $i(x)=(x,x)\in Z$, $x\in X$.
To this end we can assume that $X=Y$. We can also assume that there is a holomorphic section $\Psi$ 
of a Hermitian vector bundle $E\to Z$ defining $i(X)$ such that the induced embedding $N_X\hookrightarrow E|_{i(X)}$
is Hermitian; cf.\ Example~\ref{lokal}. By \eqref{urk} and \eqref{fullMgysin} thus
\begin{equation}\label{blaa}
i_*f^\diamond\mu=\hat c(N_X)\wedge M^\Psi\wedge\pi_2^*\mu.
\end{equation}

Our considerations are local so we can assume that $\mu=g_*\alpha$, where $g\colon V\to Y$ and $\alpha$ are as in Definition~\ref{PSdef}.
Then
\begin{equation}\label{blaaa}
\pi_2^*\mu=(\text{id}_X\times g)_*(1\otimes \alpha);
\end{equation}
cf.\ \eqref{lakan}. 
Recalling that $X=Y$, let $j\colon V\to X\times V$ be the mapping $j(v)=(g(v),v)$. Since $\Psi$ defines $i(X)$, which is the diagonal in $X\times Y$,
we have that $(\text{id}_X\times g)^*\Psi$ defines $j(V)$. Let $N_V\to j(V)$ be the normal bundle of $j(V)$ equipped with the 
Hermitian metric induced by the embedding $N_V\hookrightarrow (\text{id}_X\times g)^*E|_{j(V)}$. 
Since $\text{codim}\, j(V)=\text{dim}\, X=\text{dim}\, Y=\text{codim}\, i(X)$, it follows from Lemma~\ref{sladd} that
$N_V=(\text{id}_X\times g)^*N_X$ as Hermitian bundles. By functoriality of Chern forms thus $\hat c(N_V)=(\text{id}_X\times g)^*\hat c(N_X)$.
Hence, in view of \eqref{blaa}, \eqref{blaaa}, and \eqref{mos},
\begin{eqnarray*}
i_*f^\diamond\mu &=& (\text{id}_X\times g)_*\big((\text{id}_X\times g)^*\hat c(N_X)\wedge M^{(\text{id}_X\times g)^*\Psi}\wedge(1\otimes \alpha)\big) \\
&=&
(\text{id}_X\times g)_*\big(\hat c(N_V)\wedge M^{(\text{id}_X\times g)^*\Psi}\wedge(1\otimes \alpha)\big).
\end{eqnarray*}
By \eqref{fullMgysin} and Proposition~\ref{gyProp1} thus
\begin{equation*}
i_*f^\diamond\mu = (\text{id}_X\times g)_*j_*j^{!!}(1\otimes \alpha) = (\text{id}_X\times g)_*j_*j^{*}(1\otimes \alpha)
\end{equation*}
since $1\otimes \alpha$ is smooth. It is straightforward to check that $(\text{id}_X\times g)\circ j=i\circ g$ and that
$j^*(1\otimes\alpha)=\alpha$.
Hence, $i_*f^\diamond\mu=i_*g_*\alpha=i_*\mu$, which shows \eqref{lur}. \qed
\end{example}

\begin{proposition}\label{submersionProp}
Let $X$ be a complex manifold, $Y$ a complex Hermitian manifold, and $f\colon X\to Y$ a holomorphic mapping.
\begin{itemize}
\item[(a)] If $U\subset X$ is an open subset and $\mu\in\PS(Y)$, then  $(f^{*}\mu)|_U=f|_U^*\mu$. 

\item[(b)] Assume that $f(X)$ is a complex submanifold of $Y$ and that the induced mapping 
$\tilde f\colon X\to f(X)$ is a submersion.
Let $\iota\colon f(X)\to Y$ be the inclusion.
If $\mu\in\PS(Y)$, then $f^{*}\mu=\tilde f^*\iota^*\mu$, where $\tilde f^*$ is the standard 
pullback of currents under the submersion $\tilde f$.
\end{itemize}
\end{proposition}

\begin{proof}
(a): Recall that $N_X$ is the normal bundle of $i(X)$ in $Z=X\times Y$ equipped with the metric induced by \eqref{roffe}.
If $N_U$ is the normal bundle of $i(U)$ in $U\times Y$,
then $N_U=(N_X)|_{i(U)}$, and we equip $N_U$ with the induced metric. 
In view of Proposition~\ref{ketchup}, we have $\hat s_j(N_X,\pi_2^*\mu)|_U = \hat s_j(N_U,\pi_2^*\mu)$.
It thus follows from Definition~\ref{ister} and \eqref{gysin} that $(f^{*}\mu)|_U=f|_U^*\mu$.

(b): We can localize in both $X$ and $Y$; the statement is local in $X$ by part (a), and
in view of Proposition~\ref{kruka}, by a partition of unity in $Y$ we can assume that $\mu$ has support in any 
given open subset of $Y$.
Since $\tilde f$ is a submersion, after localization in $X$, we can assume that 
$X=X'\times X''$, where $X'$ is an open set in some $\mathbb{C}^{m'}$ and $X''$ is an open subset of $f(X)$,
such that $\tilde f\colon X'\times X''\to X''$ is the projection on the second factor. Thus, the inclusion $\iota\colon f(X)\to Y$ 
now is $X''\hookrightarrow Y$.

We will first give an expression for $\iota^\diamond\mu$.
Since we can localize in $Y$ we can assume that $Y$ is good and let $\Psi$ be a holomorphic section of 
a vector bundle $E\to Y\times Y$ such that $\Psi$ defines the diagonal $\Delta\subset Y\times Y$. 
We equip $E$ 
with a Hermitian metric so that $TY|_\Delta=N_\Delta\hookrightarrow  E|_\Delta$ is Hermitian.
Let $p_1$ and $p_2$ be the natural projections $X''\times Y\to X''$ and $X''\times Y\to Y$, respectively,
and let $j\colon X''\to X''\times Y$ be the embedding $j(x'')=(x'',\iota(x''))$.
We equip the normal bundle $N_{X''}\to j(X'')$ of $j(X'')$ in $X''\times Y$ with the metric induced by
$N_{X''}=p_2^*TY|_{j(X'')}$. The section $(\iota\times\text{id}_Y)^*\Psi$ defines $j(X'')$ and in view of Lemma~\ref{sladd},
the induced embedding $N_{X''}\hookrightarrow (\iota\times\text{id}_Y)^*E$ is Hermitian.
By \eqref{alt11} we now have
\begin{equation}\label{brun}
\iota^\diamond\mu=\iota^*\hat c(TY)\wedge (p_{1})_*(M^{(\iota\times \text{id}_Y)^*\Psi}\wedge p_2^*\mu).
\end{equation} 

We have a similar expression for $f^\diamond\mu$: The section $(f\times \text{id}_Y)^*\Psi$ defines $i(X)$ in $X\times Y$, and again in view of Lemma~\ref{sladd},
the induced embedding $N_X\hookrightarrow (f\times\text{id}_Y)^*E$ is Hermitian. By \eqref{alt11} thus
\begin{equation}\label{skal}
f^\diamond\mu=f^*\hat c(TY)\wedge (\pi_{1})_*(M^{(f\times \text{id}_Y)^*\Psi}\wedge \pi_2^*\mu).
\end{equation}
Now, since $(f\times \text{id}_Y)^*\Psi=1\otimes (\iota\times \text{id}_Y)^*\Psi$ and
$\pi_2^*\mu=1\otimes p_2^*\mu$
we have
\begin{equation}\label{banan}
M^{(f\times \text{id}_Y)^*\Psi}\wedge\pi_2^*\mu = 1\otimes \big(M^{(\iota\times \text{id}_Y)^*\Psi}\wedge p_2^*\mu\big).
\end{equation}
Moreover, since $\tilde f\colon X'\times X''\to X''$ is the standard projection,
$f^*\hat c(TY)=\tilde f^*\iota^*\hat c(TY)=1\otimes \iota^*\hat c(TY)$; cf.\ \eqref{rast}.
It thus follows by \eqref{skal}, \eqref{banan}, and \eqref{brun} that 
\begin{eqnarray*}
f^\diamond\mu &=& \big(1\otimes \iota^*\hat c(TY)\big)\wedge 
(\pi_{1})_*\big(1\otimes \big(M^{(\iota\times \text{id}_Y)^*\Psi}\wedge p_2^*\mu\big)\big) \\
&=& 1\otimes \big(\iota^*\hat c(TY)\wedge (p_{1})_*(M^{(\iota\times \text{id}_Y)^*\Psi}\wedge p_2^*\mu)\big) \\
&=& 1\otimes \iota^\diamond \mu.
\end{eqnarray*}
Taking the component of the same bidegree as $\mu$ we get $f^*\mu=1\otimes \iota^* \mu$. The right-hand side
here is the standard pullback of the current $\iota^* \mu$ under the projection 
$\tilde f\colon X'\times X''\to X''$. Part (b) of the proposition thus follows.
\end{proof}

\begin{proof}[Proof of Theorem~\ref{thm1}]
Theorem~\ref{thm1} follows from Propositions~\ref{kruka} and \ref{submersionProp}.
\end{proof}

\begin{example}\label{sib}
Let $f\colon X\to Y$ be surjective with constant fiber dimension $m-n$; recall that $m=\text{dim}\, X$ and $n=\text{dim}\, Y$.
We claim that $f^*$ then is independent of the Hermitian structure on $Y$ and that
$f^*\mu=\hat s_0(N_X,\pi_2^*\mu)$ for any $\mu\in\PS(Y)$;
recall here the currents $\hat s_j(N_X,\pi_2^*\mu)$ from Proposition~\ref{ketchup}.

To see this, we can localize in both $X$ and $Y$; cf.\ the proof of Proposition~\ref{submersionProp}.
We can thus assume that there are $\Psi$ and $E$ as in Example~\ref{lokal}.
We can also assume that $\mu=\tilde g_*\alpha$, 
where $\tilde g\colon \widetilde V\to Y$ is a holomorphic mapping, $\widetilde V$ a connected complex manifold, and
$\alpha$ a smooth compactly supported form in $\widetilde V$. Consider the fiber diagram
\eqref{eq:tomte} with $V$, $g$, and $W$ defined to be $X\times\widetilde V$, $\text{id}_X\times \tilde g$,
and $\{g^*\Psi=0\}$, respectively.
Since $f$ is surjective with constant fiber dimension $m-n$ it follows that $\text{codim}_V W=n$.
Thus, since $W=\{g^*\Psi=0\}$, it follows from \cite[Theorems~1.1, 1.2]{AESWY1} that $M_k^{g^*\Psi}=0$ if $k<n$
and that $M_n^{g^*\Psi}$ is independent of the metric. Since $\pi_2^*\mu=g_*(1\otimes \alpha)$, 
in view of \eqref{korv} and \eqref{mos} we get
$$
M_k^\Psi\wedge\pi^*_2\mu=g_*\big(M_k^{g^*\Psi}\wedge (1\otimes \alpha)\big)
=g_*\big((1\otimes \alpha)\wedge M_k^{g^*\Psi}\big).
$$ 
Hence, $M_k^\Psi\wedge\pi^*_2\mu$ vanishes if $k<n$ and is independent of the metric if 
$k=n$. It thus follows by \eqref{alt12} and Proposition~\ref{ketchup} that
$f^*$ is independent of the Hermitian structure on $Y$ and that $f^*\mu=\hat s_0(N_X,\pi_2^*\mu)$. \qed
\end{example}

\subsection{An alternative approach}\label{diagonal-alternativ}
We recall the setup; $X$ is a complex manifold, $Y$ a complex Hermitian manifold,
$f\colon X\to Y$ a holomorphic mapping, $Z=X\times Y$, and $\pi_1\colon Z\to X$ and $\pi_2\colon Z\to Y$
the natural projections. 
In this subsection we will also assume that $X$ is Hermitian and let $Z$ have the induced Hermitian structure.

Let $\iota\colon Z\to Z\times Z$ be the diagonal embedding and  $N_Z\to \iota(Z)$
the normal bundle of $\iota(Z)$ in $Z\times Z$. Then $N_Z$ is naturally isomorphic with $TZ$ and thus $N_Z$ gets an
induced Hermitian metric. The Gysin mappings $\iota^!$ and $\iota^{!!}$ are thus defined.

Recall from the introduction that the main step when defining $f^*\mu$ is to give a reasonable
meaning to the product $[i(X)]\wedge\pi_2^*\mu$, where $i(X)$ as before is the graph of $f$ in $Z$.
We have done this by applying the (local) operator(s) $M^\Psi$ to $\pi_2^*\mu$,
where $\Psi$ is a (local) holomorphic section of some Hermitian bundle
defining $i(X)$ (locally) in $Z$, and then multiplying by Chern forms.
This definition of the product $[i(X)]\wedge\pi_2^*\mu$ is not symmetric 
in $[i(X)]$ and $\pi_2^*\mu$. The objective of this section is to provide an alternative 
definition
that puts $[i(X)]$ and $\pi_2^*\mu$ on an equal footing. 
The alternative definition coincides with the product introduced in 
\cite[Definition~5.3]{AESWY2} if $\pi_2^*\mu$ is a (generalized) cycle
and $\iota(Z)$ can be defined by a global holomorphic section of some vector bundle over 
$Z\times Z$.

Before stating the main result of this subsection we notice that if 
$\Gamma\subset Z$ is a submanifold, or merely an analytic subset,  and $\mu\in\PS(Y)$, then 
$[\Gamma]\otimes \pi_2^*\mu\in\PS(Z\times Z)$.
In fact, if $\mu=g_*\alpha$,
where $g\colon V\to Y$ and $\alpha$ are as in Definition~\ref{PSdef},
then 
\begin{equation}\label{katt}
[\Gamma]\otimes \pi_2^*\mu = \tau_*(1\otimes 1\otimes\alpha)
\end{equation} 
where 
$\tau\colon \Gamma\times X\times V\to Z\times Z$, $\tau(\gamma,x,v)= (\gamma, x,g(v))$.

\begin{theorem}\label{altdef}
Let $X$ and $Y$ be complex Hermitian manifolds and
let $Z=X\times Y$ with the induced Hermitian structure.
Let $\iota\colon Z\to Z\times Z$ be the diagonal embedding
and equip the normal bundle $N_Z$ of $\iota(Z)$ with the Hermitian metric induced by
$N_Z\simeq TZ$.
Assume that $\Gamma\subset Z$ is a submanifold such that $\text{dim}\,\Gamma=\text{dim}\, X$ and
$\pi_1|_\Gamma$ is proper, where $\pi_1\colon Z\to X$ is the natural projection. Then,
with $\pi_2\colon Z\to Y$ the natural projection,
\begin{equation}\label{yster2}
(\pi_{1})_* \iota^{!} ([\Gamma]\otimes\pi_2^*\mu) = (\pi_1|_\Gamma)_* (\pi_2|_\Gamma)^{*}\mu.
\end{equation}
\end{theorem}

Notice that $(\pi_2|_\Gamma)^{*}$ depends on the Hermitian structure on $Y$ but not on the one on $X$.
The left-hand side of \eqref{yster2}, which a~priori depends on the Hermitian structures on both $X$ and $Y$,
is thus independent of the Hermitian structure on $X$.

Let us see what Theorem~\ref{altdef} means if $\Gamma=i(X)$ is the graph of the mapping $f\colon X\to Y$.
After identification of $i(X)$ and $X$ we have
$\pi_2|_\Gamma=f$
and $\pi_1|_\Gamma=\text{id}_X$, and so \eqref{yster2} becomes
\begin{equation}\label{telefjomp0}
(\pi_{1})_* \iota^{!} ([i(X)]\otimes\pi_2^*\mu) = f^*\mu.
\end{equation}
The left-hand side thus is an alternative definition of $f^*\mu$. 
The product $[i(X)]\wedge\pi_2^*\mu$ in this setting is $\iota^{!} ([i(X)]\otimes\pi_2^*\mu)$, which
clearly puts $[i(X)]$ and $\pi_2^*\mu$
on an equal footing.

To prove Theorem~\ref{altdef} we will use the following technical lemma.

\begin{lemma}\label{lma:faktor}
Let $\Sigma$ and $\Omega$ be complex manifolds and $q\colon \Sigma\to \Omega$ a holomorphic mapping. 
Let $j \colon \Sigma\to \Sigma\times \Omega$ be the graph embedding,
$N_\Sigma$ the normal bundle of $j(\Sigma)$, 
and $p\colon \Sigma\times \Omega\to \Sigma$ 
the natural projection. 
Let $\Phi_1$ be a holomorphic section of a Hermitian vector bundle $F_1\to \Sigma$ and assume that there is a holomorphic section
$\Phi_2$ of a Hermitian vector bundle
$F_2\to \Sigma\times \Omega$ defining $j(\Sigma)$ in $\Sigma\times \Omega$.  
If $N_\Sigma$ is equipped with the Hermitian metric induced by $N_\Sigma\hookrightarrow F_2|_{j(\Sigma)}$,
then for any $\mu\in\PS(\Sigma)$,
\begin{equation*}
j_*(M^{\Phi_1}\wedge\mu) = \hat c(N_\Sigma)\wedge M^{p^*\Phi_1+\Phi_2}\wedge p^*\mu.
\end{equation*}
\end{lemma}

\begin{proof}
Assume first that $\mu$ is a smooth form and that $\Phi_1$ defines a divisor $D$.
Then $p^*\Phi_1$ defines the divisor $D\times \Omega$.
The section $\Phi_2$ defines $j(\Sigma)$, which is a locally complete intersection of codimension $\text{dim}\,\Omega$.
It is clear that $D\times \Omega$ and $j(\Sigma)$ intersect properly and it follows that 
$p^*\Phi_1+\Phi_2$ defines a locally complete intersection of codimension $\text{dim}\,\Omega +1$. 
The conditions in \cite[Proposition~7.6]{AESWY1} are then satisfied and the conclusion is that
\begin{equation}\label{lut}
M^{p^*\Phi_1+\Phi_2}=M^{p^*\Phi_1}\wedge M^{\Phi_2}.
\end{equation} 
Since $\Phi_2$ defines $j(\Sigma)$, in view of Example~\ref{grillad} we have 
$M^{\Phi_2}=\hat s(N_\Sigma)\wedge j_*1$. Hence, by \eqref{lut},
\begin{equation}\label{lutfisk}
M^{p^*\Phi_1+\Phi_2}=M^{p^*\Phi_1}\wedge \hat s(N_\Sigma)\wedge j_*1.
\end{equation} 
Multiplying by $\hat c(N_\Sigma)\wedge p^*\mu$ and using \eqref{korv}, \eqref{cs1}, and \eqref{mos} we get
\begin{eqnarray*}
\hat c(N_\Sigma)\wedge M^{p^*\Phi_1+\Phi_2} \wedge p^*\mu &=&
\hat c(N_\Sigma)\wedge p^*\mu \wedge \hat s(N_\Sigma) \wedge M^{p^*\Phi_1}\wedge j_*1 \\
&=&
p^*\mu\wedge M^{p^*\Phi_1}\wedge j_*1
=j_*(j^*p^*\mu\wedge M^{j^*p^*\Phi_1}).
\end{eqnarray*}
Since $j^*p^*=(p\circ j)^*=\text{id}^*_\Sigma$ the lemma thus follows in the case when $\mu$ is smooth and 
$\Phi_1$ defines a divisor.

We now consider the general case. The statement is local in $\Sigma\times \Omega$ so we can assume that
$\mu=g_*\alpha$, where $g\colon V\to \Sigma$ and $\alpha$ are as in Definition~\ref{PSdef}.
Possibly after a modification of $V$, by Hironaka's theorem we can also assume that $g^*\Phi_1$ 
defines a divisor.

Let $\tilde j\colon V\hookrightarrow V\times\Omega$ be the embedding $\tilde j(v)=(v,q\circ g(v))$ and let
$\tilde p\colon V\times\Omega \to V$ be the natural projection.
Notice that $(g\times \text{id}_\Omega)^*\Phi_2$ defines $\tilde j(V)$ in $V\times\Omega$.
By the first part of the proof thus
\begin{equation}\label{smurf}
\tilde j_*\big(M^{g^*\Phi_1}\wedge\alpha\big) = 
\hat c(N_V)\wedge M^{\tilde p^*g^*\Phi_1+(g\times \text{id}_\Omega)^*\Phi_2}\wedge (\alpha\otimes 1),
\end{equation}
where $N_V$ is the normal bundle of $\tilde j(V)$ equipped with the metric induced by 
$N_V\hookrightarrow (g\times \text{id}_\Omega)^*F_2|_{\tilde j(V)}$. The lemma will follow by applying 
$(g\times \text{id}_\Omega)_*$ to \eqref{smurf}.

Since $(g\times \text{id}_\Omega)\circ \tilde j=j\circ g$, in view of \eqref{mos} we have
\begin{equation}\label{smurf2}
(g\times \text{id}_\Omega)_*\tilde j_*\big(M^{g^*\Phi_1}\wedge\alpha\big) =
j_*g_*\big(M^{g^*\Phi_1}\wedge\alpha\big) = j_*\big(M^{\Phi_1}\wedge\mu\big).
\end{equation}
To calculate $(g\times \text{id}_\Omega)_*$ of the right-hand side of \eqref{smurf}, notice first that
by Lemma~\ref{sladd} we have an isometry $N_V=(g\times \text{id}_\Omega)^*N_\Sigma$. We can thus replace
$\hat c(N_V)$ in the right-hand side of \eqref{smurf} by $(g\times \text{id}_\Omega)^*\hat c(N_\Sigma)$.
Since $g\circ \tilde p=p\circ (g\times \text{id}_\Omega)$ we can also replace $\tilde p^*g^*\Phi_1$
by $(g\times \text{id}_\Omega)^*p^*\Phi_1$. 
After doing this, using \eqref{projformel} and \eqref{mos} we get
\begin{multline}\label{smurf3}
(g\times \text{id}_\Omega)_*\Big(
\hat c(N_V)\wedge M^{\tilde p^*g^*\Phi_1+(g\times \text{id}_\Omega)^*\Phi_2}\wedge (\alpha\otimes 1)\Big)
= \\
\hat c(N_\Sigma)
\wedge M^{p^* \Phi_1+\Phi_2}\wedge (g\times \text{id}_\Omega)_*(\alpha\otimes 1).
\end{multline}
Since $(g\times \text{id}_\Omega)_*(\alpha\otimes 1)=\mu\otimes 1=p^*\mu$
the lemma follows by \eqref{smurf}, \eqref{smurf2}, and \eqref{smurf3}.
\end{proof}

\begin{proof}[Proof of Theorem~\ref{altdef}]
We will show that
\begin{equation}\label{rutt}
(\pi_{1})_* \iota^{!!}([\Gamma]\otimes\pi_2^*\mu) = (\pi_1|_\Gamma)_*(\pi_2|_\Gamma)^\diamond\mu.
\end{equation}
The theorem then follows by taking the component of the same bidegree as $\mu$. Let 
\begin{equation*}
T:=(\pi_2|_\Gamma)^\diamond\mu.
\end{equation*}
Then $(\pi_1|_\Gamma)_*T$ is the right-hand side of \eqref{rutt}. On the other hand,
we can factorize $\pi_1|_\Gamma$ as follows
\begin{equation*}
\Gamma \stackrel{j_1}{\longrightarrow} \Gamma\times Y \stackrel{j_2}{\longrightarrow}
\Gamma\times X\times Y \stackrel{j_3}{\longrightarrow}
Z\times X\times Y \stackrel{\pi}{\longrightarrow} X,
\end{equation*}
\begin{multline*}
j_1(\gamma)=(\gamma,\pi_2(\gamma)),\quad j_2(\gamma,y)=(\gamma,\pi_1(\gamma),y), \quad
j_3(\gamma,x,y)=(\pi_1(\gamma),
\pi_2(\gamma); x,y),\\
\pi(z;x,y)=x.
\end{multline*}
We thus have that
$
(\pi_1|_\Gamma)_*T = \pi_*(j_3)_*(j_2)_*(j_1)_*T
$
and we claim  that 
\begin{equation}\label{rutter}
\pi_*(j_3)_*(j_2)_*(j_1)_*T = (\pi_{1})_* \iota^{!!}([\Gamma]\otimes\pi_2^*\mu),
\end{equation} 
which then proves \eqref{rutt}.

It remains to show \eqref{rutter}. The claimed identity \eqref{rutter} can be checked locally
in $X$ and by a partition of unity we can assume that $\mu$ has compact support in any given open subset of $Y$.
We can thus assume that both $Y$ and $X$ are good and let $\Psi_1$ and $\Psi_2$ be holomorphic sections of vector bundles
$E_1\to Y\times Y$ and $E_2\to X\times X$ defining the diagonals $\Delta_Y\subset Y\times Y$ and 
$\Delta_X\subset X\times X$, respectively.
We equip $E_1$ and $E_2$ with Hermitian metrics such that $TY\simeq N_{\Delta_Y}\hookrightarrow E_1|_{\Delta_Y}$
and $TX\simeq N_{\Delta_X}\hookrightarrow E_2|_{\Delta_X}$ are Hermitian embeddings.

Let $\Phi_1=(\pi_2|_\Gamma\times \text{id}_Y)^*\Psi_1$ and $F_1=(\pi_2|_\Gamma\times \text{id}_Y)^*E_1$.
Then $\Phi_1$ defines $j_1(\Gamma)$ in $\Gamma\times Y$ and we let $N_\Gamma$ be the normal bundle
equipped with the metric induced by $N_\Gamma\hookrightarrow F_1|_{j_1(\Gamma)}$.
In view of Lemma~\ref{sladd}, this is the same as the metric induced by $N_\Gamma\simeq TY$.
Hence, $\hat c(N_\Gamma)=(1\otimes\hat c(TY))|_{j_1(\Gamma)}$.
Now, in view of \eqref{urk} and \eqref{rast} we have $T=(\pi_2|_\Gamma)^\diamond\mu=j_1^{!!}(1\otimes\mu)$. By \eqref{fullMgysin} thus
\begin{equation}\label{skurk}
(j_{1})_* T 
= (1\otimes\hat c(TY))\wedge M^{\Phi_1}\wedge(1\otimes \mu).
\end{equation} 

Let $F_2$ and $\Phi_2$ be the pullback of $E_2$ and $\Psi_2$, respectively,
 under the mapping $\Gamma\times X\times Y\to X\times X$, 
$(\gamma,x,y)\mapsto (\pi_1(\gamma), x)$. Then $\Phi_2$ defines $j_2(\Gamma\times Y)$, which is the graph of 
the mapping $\Gamma\times Y\to X$, $(\gamma,y)\mapsto \pi_1(\gamma)$. Let 
$p\colon \Gamma\times X\times Y\to \Gamma\times Y$ be the natural projection. Since
$1\otimes\hat c(TY)=j_2^*(1\otimes 1\otimes\hat c(TY))$ it follows by \eqref{skurk}, \eqref{projformel},
and Lemma~\ref{lma:faktor} that 
\begin{equation*}
(j_{2})_*(j_{1})_* T 
=
(1\otimes 1\otimes\hat c(TY))\wedge \hat c(N_{\Gamma\times Y})\wedge M^{p^*\Phi_1+\Phi_2}\wedge p^*(1\otimes \mu),
\end{equation*}
where the normal bundle $N_{\Gamma\times Y}$ of $j_2(\Gamma\times Y)$ is equipped with the metric induced by 
$N_{\Gamma\times Y}\hookrightarrow F_2|_{j_2(\Gamma\times Y)}$. In view of Lemma~\ref{sladd},
$
\hat c(N_{\Gamma\times Y})=(1\otimes \hat c(TX)\otimes 1)|_{j_2(\Gamma\times Y)}.
$
Thus, since $p^*(1\otimes\mu)=1\otimes 1\otimes\mu$,
\begin{equation}\label{skurk2}
(j_{2})_*(j_{1})_* T 
=
(1\otimes \hat c(TX)\otimes\hat c(TY)) \wedge M^{p^*\Phi_1+\Phi_2}\wedge (1\otimes 1\otimes\mu).
\end{equation}

From the definitions of $\Phi_1$, $\Phi_2$, $p$, and $j_3$ it is straightforward to check that
\begin{equation}\label{skurk3}
p^*\Phi_1+\Phi_2 = j_3^*(\Pi^*\Psi_1+\widetilde\Pi^*\Psi_2),
\end{equation}
where $\Pi\colon Z\times Z\to Y\times Y$ and $\widetilde\Pi\colon Z\times Z\to X\times X$
are the natural projections.
Moreover, 
since $TZ=TX\oplus TY$ we have $\hat c(TZ)=\hat c(TX)\otimes \hat c(TY)$. It follows that
\begin{equation}\label{skurk4}
1\otimes \hat c(TX)\otimes \hat c(TY) = j_3^*(1\otimes \hat c(TZ)).
\end{equation}
Replacing $p^*\Phi_1+\Phi_2$ and $1\otimes \hat c(TX)\otimes \hat c(TY)$ in \eqref{skurk2}
by the right-hand sides of \eqref{skurk3} and \eqref{skurk4}, respectively, it follows by 
\eqref{projformel} and \eqref{mos} that
$$
(j_{3})_*(j_{2})_*(j_{1})_* T =
(1\otimes \hat c(TZ))\wedge M^{\Pi^*\Psi_1+\widetilde\Pi^*\Psi_2}\wedge (j_{3})_*(1\otimes 1\otimes\mu).
$$
Since
$
(1\otimes \hat c(TZ))|_{\iota(Z)} = \hat c(N_Z)
$
and
$
(j_{3})_*(1\otimes 1\otimes\mu) = [\Gamma]\otimes\pi_2^*\mu
$
we have
\begin{equation}\label{quasimodo}
(j_{3})_*(j_{2})_*(j_{1})_* T =
\hat c(N_Z)\wedge M^{\Pi^*\Psi_1+\widetilde\Pi^*\Psi_2}\wedge ([\Gamma]\otimes\pi_2^*\mu).
\end{equation}

Now, $\Pi^*\Psi_1+\widetilde\Pi^*\Psi_2$ defines $\iota(Z)$ in $Z\times Z$ and the corresponding 
embedding $N_Z\hookrightarrow (\Pi^*E_1\oplus \widetilde\Pi^*E_2)|_{\iota(Z)}$ is Hermitian. 
In view of \eqref{quasimodo} and \eqref{fullMgysin} thus
$$
(j_{3})_*(j_{2})_*(j_{1})_* T =\iota_*\iota^{!!}([\Gamma]\otimes\pi_2^*\mu).
$$
By applying $\pi_*$, using that $\pi_*\iota_*=(\pi_{2})_*$ since $\pi\circ\iota=\pi_2$, 
thus \eqref{rutter} follows. This completes the proof.
\end{proof}

\begin{remark}\label{strudel}
As mentioned in the introduction, the alternative definition \eqref{telefjomp0} of $f^*$ 
can be generalized to define a pullback under meromorphic mappings 
and meromorphic correspondences: 

If $h\colon X\dashrightarrow Y$ is a meromorphic mapping, then the closure of the graph of $h$ is an analytic subset of $Z$, and we let 
$\Gamma$ be this analytic subset. If $\mu\in\PS(Y)$, then $[\Gamma]\otimes \pi_2^*\mu\in\PS(Z\times Z)$ and we can define
$h^*\colon \PS(Y)\to\PS(X)$ by letting
\begin{equation}\label{telefjomp}
h^*\mu := (\pi_{1})_* \iota^{!} ([\Gamma]\otimes\pi_2^*\mu).
\end{equation}

If $\Gamma$ instead is a meromorphic correspondence, i.e., $X$ and $Y$ have the same dimension $m$ and
$\Gamma$ is an effective $m$-cycle such that $\pi_1$ and $\pi_2$ restricted to each irreducible component of $|\Gamma|$
are surjective and proper, then the right-hand side of \eqref{telefjomp} makes sense and is 
a reasonable definition of a pullback of $\mu$ under the correspondence $\Gamma$. \qed
\end{remark}

\section{Compatibility with cohomology}\label{cohom}
The objective of this section is to prove Theorem~\ref{thm2}. We need a few preliminary results.
For a complex manifold $Z$ we let $H^*(Z)$ be the de Rham cohomology groups. 
Recall that $H^*(Z)$ can be defined using either closed smooth
forms or closed currents. If $\mu$ is a closed smooth form or a closed current we let 
$[\mu]_{dR}$ be the corresponding cohomology class.
If $\mu=[W]$ is the current of integration along $W\subset Z$, then we write $[W]_{dR}$ instead of $[[W]]_{dR}$.

If $g\colon Z\to \widetilde Z$ is a holomorphic mapping, 
then  the pullback of smooth forms induces a mapping $g^*\colon H^*(\widetilde Z)\to H^*(Z)$. 
If $h\colon W\to Z$ is a proper holomorphic mapping, then the pushforward of currents 
induces a mapping $h_*\colon H^*(W)\to H^{*+2d}(Z)$,
where $d=\text{dim}_\C\,  Z- \text{dim}_\C\, W$. The wedge product on smooth forms induces the cup product in $H^*(Z)$;
we will write $\wedge$ for the cup product of cohomology classes. Notice that if $\xi$ is a smooth 
closed form and $\mu$ is a closed
current, then
\begin{equation*}
[\xi\wedge\mu]_{dR} = [\xi]_{dR} \wedge [\mu]_{dR}.
\end{equation*}
We will also use that the projection formula \eqref{projformel}, with $\varphi$ and $\mu$ denoting cohomology classes, holds.

\begin{lemma}\label{gibbon}
Let $Z$ be a complex manifold and let $W\subset Z$ be an analytic subset.
If $\mu\in\mathcal{PS}(Z)$, then 
\begin{equation}\label{gibbons}
d\mathbf{1}_W\mu = \mathbf{1}_W d\mu, \quad d\mathbf{1}_{Z\setminus W}\mu = \mathbf{1}_{Z\setminus W} d\mu.
\end{equation}
If $\mathbf{1}_W\mu=0$ and $\mathbf{1}_{Z\setminus W} d\mu=0$, then $d\mu=0$.
\end{lemma}

\begin{proof}
Since $Z$ is a manifold, $\mu$ is pseudomeromorphic and so
by \cite[Proposition~3.8]{AW4} we have $\partial\mathbf{1}_W\mu = \mathbf{1}_W \partial\mu$.
Since $\mathcal{PS}(Z)$ is closed under conjugation, cf.\ \eqref{psbarps}, and the operator $\mathbf{1}_W$ is real we get
$\debar\mathbf{1}_W\mu = \mathbf{1}_W \debar\mu$. The first equality in \eqref{gibbons} thus follows.
We now get
\begin{equation*}
d\mathbf{1}_{Z\setminus W}\mu = d(\mu-\mathbf{1}_W\mu)=d\mu-\mathbf{1}_Wd\mu = \mathbf{1}_{Z\setminus W}d\mu,
\end{equation*}
which proves the second equality in \eqref{gibbons}. Moreover, if $\mathbf{1}_W\mu=0=\mathbf{1}_{Z\setminus W}d\mu$, then
\begin{equation*}
d\mu = \mathbf{1}_Wd\mu + \mathbf{1}_{Z\setminus W}d\mu = d\mathbf{1}_W\mu=0.
\end{equation*}
The last statement in the lemma is thus also proved.
\end{proof}

\begin{proposition}\label{lemuren}
Let $X$ and $Z$ be complex manifolds and $i\colon X\hookrightarrow Z$ an embedding.
Let $\pi\colon Bl_X\to Z$ be the blowup along $i(X)$, let $D\subset Bl_X$ be the exceptional divisor,
and $L\to Bl_X$ the associated line bundle.
\begin{itemize}
\item[(a)] If $\mu\in\PS(Z)$ and $\mathbf{1}_{i(X)}\mu=0$, 
then there is a unique $\nu\in\PS(Bl_X)$ such that $\pi_*\nu=\mu$ and $\mathbf{1}_D\nu=0$.
If $d\mu=0$, then $d\nu=0$.

\item[(b)] Assume that the normal bundle $N_X\to i(X)$ is equipped with a Hermitian metric
and equip $L$ with a Hermitian metric such that $L|_D\hookrightarrow \pi|_D^*N_X$ (cf.\ Lemma~\ref{sladd}) is a Hermitian embedding.
If $\nu\in\PS(Bl_X)$, $\mathbf{1}_{D}\nu=0$, and $n=\text{codim}\, i(X)$, then
\begin{equation}\label{renata}
i^!\pi_*\nu = \sum_{k=1}^n i^*\hat c_{n-k}(N_X)\wedge (\pi|_D)_* \hat s_{k-1}(L|_D,\nu).
\end{equation} 
\end{itemize}
\end{proposition}

Recall that $L|_D$ is the normal bundle of $D$ in $Bl_X$; see Proposition~\ref{ketchup} for the currents $\hat s_{k-1}(L|_D,\nu)$.

\begin{proof}
(a): The uniqueness of such a $\nu$ is clear since $\pi$ is a biholomorphism outside $D$ and $\mathbf{1}_D\nu=0$.
To show the existence it thus suffices  to see that if $U\subset Z$ is open, then there is such a $\nu$ in
$\pi^{-1}(U)$. 

In $U$ we can assume that $\mu=g_*\alpha$, where $g\colon V\to U$ 
and $\alpha$ are as in Definition~\ref{PSdef}. Since $\mathbf{1}_{i(X)}\mu=0$ we can assume that $g^{-1} i(X)\neq V$;
otherwise $\mu=0$ in view of \eqref{restr2}. 
After a modification of $V$, by Hironaka's theorem we can then assume that $g^{-1} i(X)$ is a hypersurface.
Thus, by the defining property of $Bl_X$ there is a holomorphic mapping $\tilde g\colon V\to \pi^{-1}(U)$
such that $g=\pi\circ \tilde g$. Then $\nu:=\tilde g_*\alpha\in\PS(\pi^{-1}(U))$ and 
\begin{equation*}
\mu|_U = g_*\alpha = \pi_*\tilde g_*\alpha = \pi_*\nu.
\end{equation*}
Moreover, by \eqref{restr2},
\begin{equation*}
\mathbf{1}_D\nu = \tilde g_*(\mathbf{1}_{\tilde g^{-1}(D)}\alpha) = 0
\end{equation*}
since $\tilde g^{-1}(D)=g^{-1} i(X)\neq V$.
Thus $\nu$ has the required properties in $\pi^{-1}(U)$. 

If $d\mu=0$, then clearly $\mathbf{1}_{Bl_X\setminus D}d\nu=0$, and so $d\nu=0$ by Lemma~\ref{gibbon}. Part (a) is thus
proved.

\noindent (b): We have the commutative diagram
\begin{equation}\label{eq:tomtenisse}
\xymatrix{
D  \ar[d]^{\pi|_D} \ar[r]^j & Bl_Z \ar[d]^{\pi}\\
X \ar[r]^i & Z.
}
\end{equation}

Locally in $Z$ we can assume that we have a holomorphic section
$\Psi$ of a Hermitian vector bundle $E$ such that $\Psi$ defines $i(X)$ and the embedding
$N_X\hookrightarrow E|_{i(X)}$ is Hermitian; cf.\ Example~\ref{lokal}. Then $\pi^*\Psi$ defines $D$ and since
$L|_D$ is the normal bundle of $D$ it follows by Lemma~\ref{sladd} that the metric induced on $L|_D$ by
$L|_D\hookrightarrow \pi^*E|_D$ is the same as the one induced by $L|_D\hookrightarrow \pi|_D^*N_X$.
By Proposition~\ref{ketchup}, \eqref{mos} and commutativity of \eqref{eq:tomtenisse} thus
\begin{eqnarray}\label{kropp}
i_*\hat s_{k-n}(N_X,\pi_*\nu) &=& M_k^\Psi\wedge\pi_*\nu = \pi_*(M_k^{\pi^*\Psi}\wedge\nu) =
\pi_*j_*\hat s_{k-1}(L|_D,\nu) \\
&=&
 i_*(\pi|_D)_*\hat s_{k-1}(L|_D,\nu) \nonumber
\end{eqnarray}
locally in $Z$. Since $\hat s_{k-n}(N_X,\pi_*\nu)$ and $\hat s_{k-1}(L|_D,\nu)$ are globally defined in $X$ and $D$, respectively,
\eqref{kropp} holds in $Z$.

Since $\mathbf{1}_{D}\nu=0$ it follows by \eqref{restr2} and \eqref{gurkmajo} that $M_0^\Psi\wedge\pi_*\nu=0$. 
By Proposition~\ref{ketchup} and injectivity of $i_*$ thus 
$\hat s_{-n}(N_X,\pi_*\nu)=0$. Hence, \eqref{renata} follows by \eqref{gysin}, \eqref{kropp}, and injectivity of $i_*$.
\end{proof}

In the remaining part of this section we let $f\colon X\to Y$ be a holomorphic mapping between compact complex manifolds
$X$ and $Y$ of dimensions $m$ and $n$, respectively.
We let $Z=X\times Y$ and $\pi_1\colon Z\to X$ and $\pi_2\colon Z\to Y$ be the natural projection. 
We also let $i\colon X\to Z$ be the embedding defined by $i(x)=(x,f(x))$.

\begin{lemma}\label{kohom-graflemma}
Assume that $\tau$ and $\nu$ are closed currents in $X$ and $Z$, respectively. Then
$[\tau]_{dR}=i^*[\nu]_{dR}$ if and only if $i_* [\tau]_{dR} = [i(X)]_{dR}\wedge [\nu]_{dR}$.
\end{lemma}

\begin{proof}
Let $\varphi$ be a smooth representative of $[\nu]_{dR}$. 
Then, by definition, $i^*[\nu]_{dR}=[i^*\varphi]_{dR}$. Therefore, if $[\tau]_{dR}=i^*[\nu]_{dR}$ we have
\begin{eqnarray*}
i_*[\tau]_{dR} &=& i_* [i^*\varphi]_{dR} = [i_*i^*\varphi]_{dR} = [[i(X)]\wedge \varphi]_{dR} = [i(X)]_{dR}\wedge [\varphi]_{dR} \\
&=&
[i(X)]_{dR}\wedge [\nu]_{dR}.
\end{eqnarray*}

Conversely, assume that $i_* [\tau]_{dR} = [i(X)]_{dR}\wedge [\nu]_{dR}$. Since, by the preceding calculation, 
$[i(X)]_{dR}\wedge [\nu]_{dR} = i_*[i^*\varphi]_{dR}$ it follows that
\begin{equation*}
i_*[\tau-i^*\varphi]_{dR}=0.
\end{equation*} 
This precisely means that for any smooth closed form $\xi$ in $Z$ we have
$(\tau-i^*\varphi).i^*\xi=0$. Let $\xi'$ be a smooth closed form in $X$. 
Since $\pi_1\circ i=\text{id}_X$ and $\pi_1^*\xi'$ is closed in $Z$ we get
\begin{equation*}
(\tau-i^*\varphi).\xi'=(\tau-i^*\varphi).i^*\pi_1^*\xi'=0.
\end{equation*}
Hence, $[\tau-i^*\varphi]_{dR}=0$, and thus $[\tau]_{dR}=i^*[\varphi]_{dR}=i^*[\nu]_{dR}$.
\end{proof}

\begin{lemma}\label{orangutang}
Suppose that $Y$ is good and Hermitian. Equip the normal bundle $N_X$ of $i(X)$  with the Hermitian metric
induced by \eqref{roffe}.
Let $\pi\colon Bl_X\to Z$ be the blowup along $i(X)$ and let $D\subset Bl_X$ be the exceptional divisor. 
If $\nu,\nu'\in\mathcal{PS}(Bl_X)$ are closed and such that 
$\mathbf{1}_D\nu=\mathbf{1}_D\nu'=0$ and $[\nu]_{dR}=[\nu']_{dR}$, then
\begin{equation*}
i_*[i^!\pi_*\nu]_{dR} = i_*[i^!\pi_*\nu']_{dR}.
\end{equation*}
\end{lemma}

\begin{proof}
Let $L$ be the line bundle associated with $D$ and equip $L|_D$ with the Hermitian metric induced by the embedding
$L|_D\hookrightarrow \pi|_D^*N_X$; cf.\ Lemma~\ref{sladd}. We claim that 
\begin{equation}\label{vecka}
i_*(\pi|_D)_*[\hat s_{k-1}(L|_D,\nu)]_{dR} = i_*(\pi|_D)_*[\hat s_{k-1}(L|_D,\nu')]_{dR}, \quad k\geq 1.
\end{equation}
Taking the claim for granted the lemma follows by applying $i_*$ to \eqref{renata}, use \eqref{projformel}, 
and take the cohomology class.

It remains to show the claim. Since $Y$ is good there is a holomorphic section $\Phi$ of a holomorphic vector bundle
$F$ defining the diagonal in $Y\times Y$. Then $\Psi:=(f\times \text{id}_Y)^*\Phi$
is a holomorphic section of $E:=(f\times \text{id}_Y)^*F$ defining $i(X)$.
We equip $E$ with a Hermitian metric so that the embedding $N_X\hookrightarrow E|_{i(X)}$ is Hermitian.

Now, $\pi^*\Psi$ defines $D$ and by \eqref{rov} there is an induced Hermitian metric on $L$. 
By Lemma~\ref{sladd},
on $L|_D$ this metric is the same as the one induced by
the embedding $L|_D\hookrightarrow \pi|_D^*N_X$. Let $j\colon D\to Bl_X$ be the inclusion.
In view of Propositions~\ref{ketchup} and \ref{besikta}, since $\mathbf{1}_D\nu=\mathbf{1}_D\nu'=0$,
for $k\geq 1$ we have
$$
j_*\hat s_{k-1}(L|_D,\nu)=M_k^{\pi^*\Psi}\wedge\nu = -\hat c(L^*)^k\wedge\mathbf{1}_{Bl_X\setminus D}\nu + 
d\nu_k = -\hat c(L^*)^k\wedge\nu + 
d\nu_k
$$
for some currents $\nu_k$ in $Bl_X$. For some currents $\nu_k'$ the same formula with $\nu$ and $\nu_k$ replaced by 
$\nu'$ and $\nu'_k$, respectively, holds.  Hence, since $[\nu]_{dR} = [\nu']_{dR}$,
\begin{equation}\label{veck}
j_*[\hat s_{k-1}(L|_D,\nu)]_{dR} = -c(L^*)^k\wedge [\nu]_{dR} = j_*[\hat s_{k-1}(L|_D,\nu')]_{dR}, \quad k\geq 1.
\end{equation}
By applying $\pi_*$ to \eqref{veck} and using that $\pi_*j_*= i_*(\pi|_D)_*$, cf.\ \eqref{eq:tomtenisse}, then
\eqref{vecka} follows. The claim is thus proved and the lemma follows. 
\end{proof}

We will use the following well-known facts about the blowup $\pi\colon Bl_X\to Z$
along $i(X)$, see, e.g., \cite[Chapter~4.6]{GrHa}. See \eqref{eq:tomtenisse} for the notation we will use. 
There is an exact sequence of de~Rham cohomology groups
\begin{equation}\label{exakt}
0\to H^k(Z) \longrightarrow H^k(Bl_X)\oplus H^k(X) \longrightarrow H^k(D)\to 0;
\end{equation}
the mapping $H^k(Z)\to H^k(Bl_X)\oplus H^k(X)$ is given by $a\mapsto (\pi^*a,i^*a)$ and the mapping 
$H^k(Bl_X)\oplus H^k(X)\to H^k(D)$ is given by $(a,b)\mapsto j^* a - (\pi|_D)^*b$. 
Moreover, $H^*(D)$ is an $H^*(X)$-algebra via $(\pi|_D)^*$ and as such 
it is generated by $j^* c_1(L^*)$, where $L$ is the line bundle corresponding to $D$. 
This means that if 
$\beta$ is a smooth closed $r$-form in $D$, then there are 
smooth closed $\ell$-forms $\beta_\ell$ on $X$ such that
\begin{equation}\label{trumpet}
[\beta]_{dR} = \sum_{\ell=0}^r [(\pi|_D)^*\beta_\ell]_{dR}\wedge  j^{*} c_1(L^*)^{r-\ell}.
\end{equation}
In the $H^*(X)$-algebra $H^*(D)$ there is the relation
\begin{equation}\label{rel}
\sum_{k=0}^n(\pi|_D)^*i^*c_{n-k}(N_X)\wedge j^*c_1(L^*)^k=0.
\end{equation}

\begin{theorem}\label{skon}
Let $X$ be a compact complex manifold, $Y$ a good compact complex Hermitian manifold, and $f\colon X\to Y$ a holomorphic
mapping. Let $Z=X\times Y$ and $i\colon X\to Z$ the embedding $i(x)=(x,f(x))$.
Equip the normal bundle $N_X$ of $i(X)$ with the Hermitian metric induced by \eqref{roffe}.
If $\mu\in\mathcal{PS}(Z)$ is closed, then
\begin{equation}\label{chimp}
[i^{!}\mu]_{dR} = i^*[\mu]_{dR}, \quad \quad i_*[i^{!}\mu]_{dR}=[i(X)]_{dR}\wedge [\mu]_{dR}.
\end{equation}
\end{theorem}

\begin{proof}
To begin with, notice that the two identities in \eqref{chimp} are equivalent by Lemma~\ref{kohom-graflemma}.
We have $\mu=\mathbf{1}_{i(X)}\mu + \mathbf{1}_{Z\setminus i(X)}\mu$ and by Lemma~\ref{gibbon} both 
$\mathbf{1}_{i(X)}\mu$ and 
$\mathbf{1}_{Z\setminus i(X)}\mu$ are closed. Thus, since $\text{supp}(\mathbf{1}_{i(X)}\mu)\subset i(X)$
and $\mathbf{1}_{i(X)}\mathbf{1}_{Z\setminus i(X)}\mu=0$, cf.\ \eqref{restr} and \eqref{restrV}, 
it suffices to show the theorem for a $\mu$ such that either
$\text{supp}\,\mu\subset i(X)$ or $\mathbf{1}_{i(X)}\mu=0$.

We first assume that $\text{supp}\,\mu\subset i(X)$. Then,
by Example~\ref{kaffe}, $i^!\mu=i^*\hat c_n(N_X)\wedge\hat s_{-n}(N_X,\mu_1)$ and $i_*\hat s_{-n}(N_X,\mu)=\mu$.
It is well-known that $i^*c_n(N_{X})=i^*[i(X)]_{dR}$ 
and so, in view of \eqref{projformel}, 
\begin{eqnarray*}
i_*[i^!\mu]_{dR} &=& i_*\big(i^*c_n(N_X)\wedge[\hat s_{-n}(N_X,\mu)]_{dR}\big)
=i_*\big(i^*[i(X)]_{dR}\wedge[\hat s_{-n}(N_X,\mu)]_{dR}\big) \\
&=&
[i(X)]_{dR}\wedge i_*[\hat s_{-n}(N_X,\mu)]_{dR}
=[i(X)]_{dR}\wedge [\mu]_{dR}.
\end{eqnarray*}
The theorem thus follows in case $\text{supp}\,\mu\subset i(X)$.

It remains to prove the theorem in case $\mathbf{1}_{i(X)}\mu=0$.
Let $\pi\colon Bl_X\to Z$ be the blowup along $i(X)$, $D$ the exceptional divisor, and 
$L$ the associated line bundle. Equip $L$ with a Hermitian metric such that $L|_D\hookrightarrow \pi|_D^*N_X$ is Hermitian,
cf.\ Lemma~\ref{sladd}.

Assume first that $\mu=\pi_*\hat c_1(L^*)^r$.
It is clear that the first identity in \eqref{chimp} holds if $r=0$ since in that case $\mu=1$.
For $r\geq 1$ and $k\geq 1$, in view of Example~\ref{unna},
\begin{equation*}
\hat s_{k-1}\big(L|_D,\hat c_1(L^*)^r\big) =  \hat c_1(L|_D^*)^{k-1}\wedge j^*\hat c_1(L^*)^r = j^*\hat c_1(L^*)^{r+k-1},
\end{equation*}
where $j\colon D\to Bl_X$ is the inclusion. Hence, if $\mu=\pi_*\hat c_1(L^*)^r$, in view of Proposition~\ref{lemuren} (b), \eqref{projformel}, 
and \eqref{rel} we get
\begin{eqnarray*}
[i^!\mu]_{dR} &=& \big[ \sum_{k=1}^n i^*\hat c_{n-k}(N_X)\wedge (\pi|_D)_*j^*\hat c_1(L^*)^{r+k-1}\big]_{dR} \\
&=& 
(\pi|_D)_*\big( \sum_{k=1}^n (\pi|_D)^*i^* c_{n-k}(N_X)\wedge  j^*c_1(L^*)^k\wedge j^*c_1(L^*)^{r-1}\big) \\
&=&
-(\pi|_D)_*\big( (\pi|_D)^*i^* c_{n}(N_X)\wedge j^*c_1(L^*)^{r-1}\big).
\end{eqnarray*}
Since $i^*c_n(N_X) = i^*[i(X)]_{dR}$, by \eqref{projformel}, and commutativity of \eqref{eq:tomtenisse} thus
\begin{eqnarray}\label{pannkaka}
i_*[i^!\mu]_{dR} &=& -i_*\big(i^* c_{n}(N_X)\wedge (\pi|_D)_*j^*c_1(L^*)^{r-1}\big) \nonumber \\
&=&
-[i(X)]_{dR}\wedge i_*(\pi|_D)_*j^*c_1(L^*)^{r-1}
=-[i(X)]_{dR}\wedge \pi_*j_*j^*c_1(L^*)^{r-1}.
\end{eqnarray}
However, by \eqref{pl},
\begin{equation}\label{pinnkaka}
\pi_*j_*j^*c_1(L^*)^{r-1} = \pi_*(c_1(L^*)^{r-1}\wedge [D]_{dR}) = -\pi_*c_1(L^*)^r = -[\mu]_{dR}.
\end{equation}
By \eqref{pannkaka} and \eqref{pinnkaka} thus $\mu$ satisfies the second identity in \eqref{chimp}
in case $\mu=\pi_*\hat c_1(L^*)^r$. 

Next, in view of Proposition~\ref{gyProp1} we notice that the theorem now follows  if
$\mu=\alpha_\ell\wedge \pi_*\hat c_1(L^*)^{r-\ell}$, where $\alpha_\ell$ is a smooth closed form in $Z$
and $\ell$ is an integer, $0\leq \ell\leq r$. 

Let us now show the theorem for a general $\mu$ such that $\mathbf{1}_{i(X)}\mu=0$.
In view of Proposition~\ref{lemuren} (a) there is a closed $\nu\in\mathcal{PS}(Bl_X)$ 
such that $\mathbf{1}_D\nu=0$ and $\pi_*\nu=\mu$.
Let $r$ be the degree of $\nu$.
We claim that there are smooth closed $\ell$-forms 
$\alpha_\ell$, $\ell=0,\ldots,r$, in $Z$ such that 
\begin{equation}\label{sent}
[\nu]_{dR}=[\sum_{\ell=0}^r\pi^*\alpha_\ell\wedge\hat c_1(L^*)^{r-\ell}]_{dR}.
\end{equation}
Taking the claim for granted for the moment we can finish the proof:
By Lemma~\ref{orangutang} and \eqref{projformel} we then have
\begin{equation*}
i_*[i^!\pi_*\nu]_{dR} = i_*[i^!\pi_* \sum_{\ell=0}^r\pi^*\alpha_\ell\wedge\hat c_1(L^*)^{r-\ell}]_{dR}
= i_*[i^!\sum_{\ell=0}^r (\alpha_\ell\wedge\pi_*\hat c_1(L^*)^{r-\ell})]_{dR}.
\end{equation*}
Hence, since \eqref{chimp} holds for each $\alpha_\ell\wedge\pi_*\hat c_1(L^*)^{r-\ell}$, in view of \eqref{projformel} and \eqref{sent}, 
\begin{eqnarray*}
i_*[i^!\mu]_{dR} &=& i_*[i^!\pi_*\nu]_{dR} 
=[i(X)]_{dR}\wedge \big[\sum_{\ell=0}^r \alpha_\ell\wedge\pi_*\hat c_1(L^*)^{r-\ell}\big]_{dR} \\
&=& [i(X)]_{dR}\wedge \pi_*\big[\sum_{\ell=0}^r \pi^*\alpha_\ell\wedge\hat c_1(L^*)^{r-\ell}\big]_{dR} \\
&=&
[i(X)]_{dR}\wedge \pi_* [\nu]_{dR} = [i(X)]_{dR}\wedge [\mu]_{dR},
\end{eqnarray*}
which shows \eqref{chimp}.

It remains to prove the claim.
Let $\alpha$ be a smooth closed $r$-form on $Bl_X$ such that $[\alpha]_{dR}=[\nu]_{dR}$.
Let $\beta=j^*\alpha$, and take smooth closed $\ell$-forms $\beta_\ell$ on 
$X$ such that \eqref{trumpet} holds. 
Letting $\alpha_\ell=\pi_1^*\beta_\ell$, where $\pi_1\colon Z\to X$ is the projection on the first factor, 
we get
$\beta_\ell=i^*\alpha_\ell$ since $\pi_1\circ i=\text{id}_X$. Since $i\circ \pi|_D=\pi\circ j$ thus 
$(\pi|_D)^*\beta_\ell=j^*\pi^*\alpha_\ell$; cf.\ \eqref{eq:tomtenisse}.
It thus follows by \eqref{trumpet} that
\begin{equation*}
j^*\big[\alpha - \sum_{\ell=0}^r\pi^*\alpha_\ell\wedge \hat c_1(L^*)^{r-\ell}\big]_{dR}=
\big[\beta - \sum_{\ell=0}^r (\pi|_D)^*\beta_\ell\wedge j^* \hat c_1(L^*)^{r-\ell}\big]_{dR}=0.
\end{equation*} 
Hence, since \eqref{exakt} is exact we can add a smooth closed $r$-form in $Z$ to $\alpha_r$ and get
\begin{equation*}
\big[\alpha -
\sum_{\ell=0}^r\pi^*\alpha_\ell\wedge \hat c_1(L^*)^{r-\ell}\big]_{dR}=0.
\end{equation*}
Since $[\alpha]_{dR} = [\nu]_{dR}$ this proves the claim and concludes the proof.
\end{proof}

\begin{proof}[Proof of Theorem~\ref{thm2}]
If $\mu\in\mathcal{PS}(Y)$ is closed, then $\pi_2^*\mu=1\otimes \mu\in\mathcal{PS}(Z)$ is closed;
recall that $\pi_2\colon Z\to Y$ is the projection on the second factor.
Using $\pi_2^*\mu$ 
as $\mu$ in Theorem~\ref{skon} we get
\begin{equation*}
[i^{!}\pi_2^*\mu]_{dR} = i^*[\pi_2^*\mu]_{dR}.
\end{equation*}
The left-hand side is $[f^*\mu]_{dR}$ by Definition~\ref{ister}. 
Since $\pi_2$ is a simple projection and since $f=\pi_2\circ i$, the right-hand side is 
$i^*\pi_2^*[\mu]_{dR}=f^*[\mu]_{dR}$.
\end{proof}

\begin{remark}
If there where a Poincar\'e lemma for $\PS$-currents on a complex manifold $Z$, then $H^*(Z)$ could be defined using closed  $\PS$-currents.
Theorem~\ref{thm2} would then follow by \eqref{sov2}. 
\end{remark}

\section{The obstruction to functoriality}\label{obstruktion}
We begin with an example showing that our pullback is not functorial in general.

\begin{example}\label{burger}
Let $g\colon \C^2\to \C^2$, $g(v_1,v_2)=(v_1v_2,v_2^2)$, and let $\alpha$ be a smooth function with compact support in $\C^2$. Then $\mu:=g_*\alpha\in\PS(\mathbb{C}^2)$ and has bidegree $(0,0)$.
Let 
$$
f\colon \{\text{pt}\}\to\C^2, \quad f(\text{pt})=0,
$$ 
be inclusion of a point. 
The mapping $f$ can be factorized
as $f=f_2\circ f_1$, where 
$$
f_1\colon \{\text{pt}\}\to\C, \,\, f_1(\text{pt})=0, \quad \text{and} \quad
f_2\colon \C\to\C^2, \,\, f_2(w)=(w,w^2).
$$
We claim that if $\mathbb{C}^2$ and $\mathbb{C}$ are equipped with their standard Hermitian structures, then
\begin{equation}\label{bombay}
f^*\mu=f^\diamond \mu = 2\alpha(0,0) \quad \text{and}\quad 
f_1^*f_2^*\mu = f_1^{\diamond}f_2^{\diamond}\mu = \alpha(1,0)+\alpha(-1,0),
\end{equation}
which shows that 
$f^*\mu\neq f_1^*f_2^*\mu$ and $f^{\diamond}\mu\neq f_1^{\diamond}f_2^{\diamond}\mu$
if $2\alpha(0,0)\neq\alpha(1,0)+\alpha(-1,0)$.

To begin with, notice that $f^*\mu=f^\diamond\mu$ for degree reasons. Let us calculate $f^\diamond\mu$.
We consider $z=(z_1,z_2)$ as a section of the trivial Hermitian rank-$2$ bundle 
over $\{\text{pt}\}\times \C^2_z$. This section 
defines the graph of $f$ in $\{\text{pt}\}\times\mathbb{C}^2$.
Let $\pi_1\colon \{\text{pt}\}\times\C^2\to \{\text{pt}\}$ and $\pi_2\colon \{\text{pt}\}\times\C^2\to \C^2$
be the standard projections.
We will make the identification $\{\text{pt}\}\times\C^2\simeq \C^2$ and then
these projections are identified with $\C^2\to\{\text{pt}\}$ and 
$\text{id}_{\C^2}$, respectively. In particular, $\pi_2^*\mu=\mu$.
Moreover, in view of \eqref{gurkmajo} and the Dimension principle, $M_0^z\wedge\mu=0=M_1^z\wedge\mu$.
Since $\hat c(T\C^2)=1$, it now follows by \eqref{alt11} that
\begin{equation}\label{vem}
f^{\diamond}\mu = f^*\hat c(T\C^2)\wedge (\pi_{1})_*(M^{z}\wedge \pi_2^*\mu)
=(\pi_{1})_*(M_2^{z}\wedge \mu).
\end{equation}
To compute the right-hand side, notice that $\mu=g_*\alpha = g_*p_*p^*\alpha$,
where $p\colon Bl_0\mathbb{C}^2\to\mathbb{C}^2$ is the blowup of $0$.
The section $p^*g^*z$ defines the divisor $2D+H$, where $D\subset Bl_0\C^2$
is the exceptional divisor and $H$ is the strict transform of $\{v_2=0\}$. 
If $L$ is the line bundle associated with $2D+H$, then $L\hookrightarrow Bl_0\C^2\times \C^2$, cf.\ \eqref{rov},
and we equip $L$ with the induced Hermitian metric.
In view of \eqref{mos} and Example~\ref{grillad} (ii), since $p^*\alpha$ is smooth, it follows that
\begin{equation}\label{dalen}
M_2^{z}\wedge \mu = g_*p_*(M_2^{p^*g^*z}\wedge p^*\alpha)
=g_*p_*\big(\hat c_1(L^*)\wedge(2[D]+[H])\wedge p^*\alpha\big).
\end{equation}
A direct calculation in the standard local coordinates on $Bl_0\C^2$
shows that $\hat c_1(L^*)\wedge [H]=0$ and $\hat c_1(L^*)\wedge [D]=\omega\wedge [D]$, where $\omega$ is the
Fubini--Study form on $D\simeq\mathbb{P}^1$. By \eqref{vem} and \eqref{dalen} thus
$$
f^{\diamond}\mu = (\pi_{1})_*g_*p_*\big(2\omega\wedge [D]\wedge p^*\alpha\big)
=(\pi_{1})_*g_*(2[0]\cdot\alpha(0)) = 2\alpha(0)(\pi_{1})_*[0]=2\alpha(0).
$$
The first two equalities in \eqref{bombay} now follow.

Let us now calculate $f_2^{\diamond}\mu$.
Let $p_1$ and  $p_2$ be the standard projections from $\C_w\times\C^2_z$ to the first and second 
factor, respectively. We consider $f_2(w)-z$ as a section of the trivial Hermitian rank-2 bundle over $\C_w\times\C^{2}_z$.
Then $f_2(w)-z$ defines the graph of $f_2$ in $\C_w\times\C_z^2$.
In a similar way as above, since $p_2^*\mu=(\text{id}_\C \times g)_*(1\otimes\alpha)$,
\begin{equation}\label{piazza}
f_2^{\diamond}\mu = (p_{1})_*(M_2^{f_2(w)-z}\wedge p_2^*\mu)
=(p_{1})_* (\text{id}_\C \times g)_* (M_2^{f_2(w)-g(v)}\wedge (1\otimes \alpha)).
\end{equation}
We have that $f_2(w)-g(v)=(w-v_1v_2, w^2-v_2^2)$, which defines a locally complete intersection 
ideal $\J$
in $\C_w\times \C_v^2$. The fundamental cycle $\mathcal{Z}_\J$ of $\J$ is the proper intersection of the divisors
$\text{div}(w-v_1v_2)$ and $\text{div}(w^2-v_2^2)$. A standard calculation gives
$$
\text{div}(w^2-v_2^2)\cdot \text{div}(w-v_1v_2) = 2[w=v_2=0] + [w=v_2, v_1=1] + [w=-v_2, v_1=-1].
$$
Thus, in view of \eqref{piazza}, \eqref{korv}, and Example~\ref{grillad}, 
$$
f_2^{\diamond}\mu =p_{1*} (\text{id}_\C \times g)_*
\big((1\otimes\alpha)\wedge(2[w=v_2=0] + [w=v_2, v_1=1] + [w=-v_2, v_1=-1])\big).
$$
Since $p_{1} \circ (\text{id}_\C \times g)$ is the natural projection $\C_w\times\C^2_v\to\C_w$
it follows that
\begin{equation*}
f_2^{\diamond}\mu 
=\alpha(1,w)+\alpha(-1,-w).
\end{equation*}
For degree reasons we have $f_2^{\diamond}\mu=f_2^*\mu$. 

Since $\alpha(1,w)+\alpha(-1,-w)$ is smooth
we have $f_1^{\diamond}(\alpha(1,w)+\alpha(-1,-w))=f_1^*(\alpha(1,w)+\alpha(-1,-w))=\alpha(1,0)+\alpha(-1,0)$. 
Hence, the last two equalities in \eqref{bombay} follow. \qed
\end{example}

Next we give a sufficient condition in terms of $M$-operators for functoriality of the $\diamond$-pullback.
On can formulate a similar condition for our $*$-pullback but it becomes a bit more technical so we omit giving the details.

Recall that we say that a complex manifold $X$ is good if there is a holomorphic section $\Phi$ of a vector bundle $F\to X\times X$
defining the diagonal $\Delta$ in $X\times X$. If in addition $X$ is Hermitian we say that such a $\Phi$ is a 
\emph{Hermitian defining section of the diagonal} if $F$ is equipped with a Hermitian metric such that the induced embedding $TX=N_\Delta\hookrightarrow F|_\Delta$,
where $N_\Delta$ is the normal bundle of $\Delta$, is Hermitian.

\begin{proposition}\label{korken}
Let $f_1\colon X_1\to X_2$ and $f_2\colon X_2\to Y$ be holomorphic mappings between complex manifolds.
Assume that $X_2$ and $Y$ are good and Hermitian. Let $\Phi$ and $\Psi$ be Hermitian defining sections
of the diagonals in $X_2\times X_2$ and $Y\times Y$, respectively. 
If $\mu\in\PS(Y)$ and if
$$
M^{\Phi(f_1(x_1),x_2)}\wedge M^{\Psi(f_2(x_2),y)}\wedge (1\otimes 1\otimes\mu)
=M^{\Phi(f_1(x_1),x_2)+\Psi(f_2\circ f_1(x_1),y)}\wedge(1\otimes 1\otimes\mu)
$$
in $X_1\times X_2\times Y$,
then $f_1^\diamond f_2^\diamond\mu=(f_2\circ f_1)^\diamond\mu$.
\end{proposition}

\begin{proof}
Let $\beta = f_1^*\hat c(TX_2)\otimes f_2^*\hat c(TY)\otimes 1$, which is a  smooth form in $X_1\times X_2\times Y$.
Let $\Pi\colon X_1\times X_2\times Y\to X_1$ be the natural projection.
We will show
\begin{equation}\label{3star}
\Pi_*\big(\beta\wedge M^{\Phi(f_1(x_1),x_2)}\wedge M^{\Psi(f_2(x_2),y)}\wedge (1\otimes 1\otimes\mu)\big) 
= f_1^\diamond f_2^\diamond\mu,
\end{equation}
\begin{equation}\label{2star}
\Pi_*\big(\beta\wedge M^{\Phi(f_1(x_1),x_2)+\Psi(f_2\circ f_1(x_1),y)}\wedge(1\otimes 1\otimes\mu)\big) = (f_2\circ f_1)^\diamond\mu.
\end{equation}
From \eqref{2star} and \eqref{3star} the proposition immediatly follows.

To prove \eqref{3star} we factorize $\Pi$ as
$$
X_1\times X_2\times Y \stackrel{\Pi_2}{\longrightarrow}
X_1\times X_2 \stackrel{\Pi_1}{\longrightarrow} X_1.
$$
Notice that $\Pi_2=\text{id}_{X_1}\times\pi_2$, where $\pi_2\colon X_2\times Y\to X_2$ is the natural projection.
For notational convenience, let
$$
T=M^{\Phi(f_1(x_1),x_2)}\wedge M^{\Psi(f_2(x_2),y)}\wedge (1\otimes 1\otimes\mu).
$$
By a small abuse of notation we have $\Phi(f_1(x_1),x_2)=\Pi_2^*\Phi(f_1(x_1),x_2)$ and so
$$
\beta\wedge T = 
\Pi_2^*\Pi_1^*f_1^*\hat c(TX_2)\wedge M^{\Pi_2^*\Phi(f_1(x_1),x_2)}\wedge
\big(1\otimes (\pi_2^*f_2^*\hat c(TY)\wedge M^{\Psi(f_2(x_2),y)}\wedge(1\otimes\mu))\big).
$$
By applying $(\Pi_2)_*$ and using \eqref{projformel}, \eqref{mos}, and that $(\Pi_2)_*=(\text{id}_{X_1}\times\pi_2)_*$, in view of  
\eqref{alt11} we get
$$
(\Pi_2)_*(\beta\wedge T) =
\Pi_1^*f_1^*\hat c(TX_2)\wedge M^{\Phi(f_1(x_1),x_2)} \wedge
\big(1\otimes f_2^\diamond\mu\big).
$$
Since $\Pi_*=(\Pi_1)_*(\Pi_2)_*$, by applying $(\Pi_1)_*$ now \eqref{3star} follows in view of \eqref{alt11} and 
\eqref{projformel}.

Let us now show \eqref{2star}. Let $f=f_2\circ f_1$. We first check that 
\begin{equation}\label{lammlar}
j^*\beta = j^*\big((f^*\hat c(TY)\otimes 1\otimes 1)\wedge\hat c(N)\big),
\end{equation}
where 
$$
j\colon X_1\times Y \to X_1\times X_2\times Y, \quad j(x_1,y) = (x_1,f_1(x_1),y)
$$
and $N\to j(X_1\times Y)$ is the normal bundle of $j(X_1\times Y)$ in $X_1\times X_2\times Y$ 
equipped with the metric induced by the natural isomorphism $N\simeq TX_2$.
Then $j^*\hat c(N) = j^*(1\otimes\hat c(TX_2)\otimes 1)$.
From the special form of $j$ we have 
$j^*(a\otimes b\otimes c) = (a\wedge f_1^*b)\otimes c$ if $a$, $b$, and $c$
are smooth forms on $X_1$, $X_2$, and $Y$, respectively. Thus,
\begin{equation}\label{grizzly}
j^*(f_1^*\hat c(TX_2)\otimes 1\otimes 1)=
f_1^*\hat c(TX_2)\otimes 1=
j^*(1\otimes\hat c(TX_2)\otimes 1) = j^*\hat c(N).
\end{equation}
Moreover,
\begin{equation}\label{och}
j^*(1\otimes f_2^*\hat c(TY)\otimes 1)=
f_1^*f_2^*\hat c(TY)\otimes 1=
f^*\hat c(TY)\otimes 1=
j^*(f^*\hat c(TY)\otimes 1\otimes 1).
\end{equation}
Now \eqref{lammlar} follows from \eqref{grizzly} and \eqref{och}.

Let 
$$
\widetilde T= M^{\Phi(f_1(x_1),x_2)+\Psi(f(x_1),y)}\wedge(1\otimes 1\otimes\mu).
$$
By our assumption in the proposition of course $\widetilde T=T$, but \eqref{2star} holds without this assumption
so we here distinguish between $\widetilde T$ and $T$.
In view of Lemma~\ref{senf}, $\widetilde T=j_*\nu$ for some ($\PS$-)current $\nu$. By \eqref{lammlar} and 
\eqref{projformel} thus
$$
\beta\wedge\widetilde T = (f^*\hat c(TY)\otimes 1\otimes 1)\wedge\hat c(N)\wedge\widetilde T.
$$
Hence, since $\Phi(f_1(x_1),x_2)$ defines $j(X_1\times Y)$ in $X_1\times X_2\times Y$ it follows by
Lemma~\ref{lma:faktor} and \eqref{projformel}  that
\begin{equation}\label{luncha}
\beta\wedge\widetilde T = 
j_*\big(j^*(f^*\hat c(TY)\otimes 1\otimes 1)\wedge M^{\Psi(f(x_1),y)}\wedge(1\otimes\mu)\big);
\end{equation}
Lemma~\ref{lma:faktor} is indeed applicable since, in view of Lemma~\ref{sladd}, the metric we have on $N$ is the same as the 
one induced by the embedding $N\hookrightarrow h^*F|_{j(X_1\times Y)}$, where $h(x_1,x_2,y)=(f_1(x_1),x_2)$
and $F\to X_2\times X_2$ is the Hermitian vector bundle of which $\Phi$ is a section.
Let $P\colon X_1\times Y\to X_1$ be the natural projection. Then $P_*=\Pi_* j_*$ and moreover
$P^*f^*\hat c(TY)=j^*(f^*\hat c(TY)\otimes 1\otimes 1)$. Hence, by applying $\Pi_*$ to \eqref{luncha}, in view of \eqref{alt11}
we get
\begin{eqnarray*}
\Pi_*(\beta\wedge\widetilde T) &=& 
\Pi_*j_*\big(P^*f^*\hat c(TY)\wedge M^{\Psi(f(x_1),y)}\wedge(1\otimes\mu)\big) \\
&=&
P_*\big(P^*f^*\hat c(TY)\wedge M^{\Psi(f(x_1),y)}\wedge(1\otimes\mu)\big)
=f^\diamond \mu.
\end{eqnarray*}
This proves \eqref{2star} and the proposition follows.
\end{proof}

Whether the assumed $M$-identity in Proposition~\ref{korken} holds
depends on how singular $\mu$ is compared to the composition $f_2\circ f_1$.
To illustrate this, assume that $\mu=g_*\alpha$, where $g\colon V\to Y$ and $\alpha$ are as in Definition~\ref{PSdef}.
Then one can show that if the fiber product $V':=X_2\times_Y V$ is a locally complete intersection in $X_2\times V$
and the fiber product $X_1\times_{X_2} V'$ is a locally complete intersection in $X_1\times V'$, then the assumed identity
in Proposition~\ref{korken} holds.
From this it follows in particular that if $\mu$ in Proposition~\ref{korken} is a smooth form, 
so that we can take $V=Y$, then the $M$-identity in the proposition holds.

We conclude by a comment about the assumption  in Proposition~\ref{korken}  that $X_2$ and $Y$ are good.
The conclusion $f_1^\diamond f_2^\diamond \mu = (f_2\circ f_1)^\diamond\mu$ of Proposition~\ref{korken} is a local statement in $X$ so
we can replace $X$ by a small neighborhood of a given $x\in X$.
In view of Proposition~\ref{submersionProp}(a) one can 
then replace $X_2$ and $Y$ in Proposition~\ref{korken} by small neighborhoods of $f_1(x)$ and $f_2(f_1(x))$, respectively. Locally any complex manifold is good,
cf.\ the proof of Proposition~\ref{ketchup}. The assumption that  $X_2$ and $Y$ are good thus becomes superfluous
after these localizations.

\end{document}